\documentclass[oneside,english]{amsart}
\usepackage[T1]{fontenc}
\usepackage[latin9]{inputenc}
\usepackage{geometry}
\usepackage{color}
\usepackage{babel}
\usepackage{mathrsfs}
\usepackage{amstext}
\usepackage{amsthm}
\usepackage{amssymb}
\usepackage[all]{xy}
\usepackage[unicode=true,pdfusetitle,
 bookmarks=true,bookmarksnumbered=false,bookmarksopen=false,
 breaklinks=false,pdfborder={0 0 1},backref=false,colorlinks=true]
 {hyperref}
\usepackage{breakurl}

\makeatletter
\numberwithin{equation}{section}
\numberwithin{figure}{section}
 \theoremstyle{definition}
 \newtheorem*{defn*}{\protect\definitionname}
  \theoremstyle{plain}
  \newtheorem*{thm*}{\protect\theoremname}
\theoremstyle{plain}
\newtheorem{thm}{\protect\theoremname}[section]
  \theoremstyle{definition}
  \newtheorem{defn}[thm]{\protect\definitionname}
  \theoremstyle{plain}
  \newtheorem{lem}[thm]{\protect\lemmaname}
  \theoremstyle{remark}
  \newtheorem{rem}[thm]{\protect\remarkname}
  \theoremstyle{plain}
  \newtheorem{prop}[thm]{\protect\propositionname}
  \theoremstyle{definition}
  \newtheorem{example}[thm]{\protect\examplename}
  \theoremstyle{definition}
  \newtheorem{xca}[thm]{\protect\exercisename}
  \theoremstyle{plain}
  \newtheorem{question}[thm]{\protect\questionname}
  \theoremstyle{plain}
  \newtheorem{cor}[thm]{\protect\corollaryname}
  \theoremstyle{definition}
  \newtheorem{condition}[thm]{\protect\conditionname}

\usepackage[all]{xy}
\usepackage{url}

\makeatother

  \providecommand{\conditionname}{Condition}
  \providecommand{\corollaryname}{Corollary}
  \providecommand{\definitionname}{Definition}
  \providecommand{\examplename}{Example}
  \providecommand{\exercisename}{Exercise}
  \providecommand{\lemmaname}{Lemma}
  \providecommand{\propositionname}{Proposition}
  \providecommand{\questionname}{Question}
  \providecommand{\remarkname}{Remark}
  \providecommand{\theoremname}{Theorem}
\providecommand{\theoremname}{Theorem}

\begin{document}

\title{Correspondences without a Core}

\author{Raju Krishnamoorthy}
\begin{abstract}
We study the formal properties of correspondences of curves without
a core, focusing on the case of étale correspondences. The motivating
examples come from Hecke correspondences of Shimura curves. Given
a correspondence without a core, we construct an infinite graph $\mathcal{G}_{gen}$
together with a large group of ``algebraic'' automorphisms $A$.
The graph $\mathcal{G}_{gen}$ measures the ``generic dynamics''
of the correspondence. We construct specialization maps $\mathcal{G}_{gen}\rightarrow\mathcal{G}_{phys}$
to the ``physical dynamics'' of the correspondence. We also prove
results on the number of bounded étale orbits, in particular generalizing
a recent theorem of Hallouin and Perret. We use a variety of techniques:
Galois theory, the theory of groups acting on infinite graphs, and
finite group schemes.

\tableofcontents{}
\end{abstract}

\maketitle

\section{Introduction}

\let\thefootnote\relax\footnotetext{Freie Universität Berlin}\let\thefootnote\relax\footnotetext{raju@math.columbia.edu}

In \cite{mochizuki1998correspondences}, Mochizuki proved that if
an étale correspondence of complex hyperbolic curves

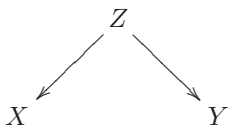
\[
\xymatrix{ & Z\ar[dl]\ar[dr]\\
X &  & Y
}
\]
has generically unbounded dynamics, then $X$, $Y$, and $Z$ are
all Shimura curves. Mochizuki uses a highly non-trivial result of
Margulis \cite{margulis1991discrete}, which characterizes Shimura
curves via properties of discrete subgroups of $PSL_{2}(\mathbb{R})$.

The most basic examples of Shimura varieties are the modular curves,
parametrizing elliptic curves with level structure. A slightly less
familiar example comes from moduli spaces of \emph{fake elliptic curves};
these Shimura varieties are projective algebraic curves. It turns
out that the modular curves are the only non-compact Shimura curves.
See Deligne \cite{deligne1979shimura} for a general introduction
to Shimura varieties.

In general, Shimura varieties are quasi-projective algebraic varieties
defined over $\overline{\mathbb{Q}}$ \cite{borovoi1982langlands,deligne1979shimura,milne1982conjugates,milne1983action}.
Recent work of Kisin \cite{kisin2010integral} shows that Shimura
varieties of abelian type have natural integral models, which opens
up the possibility of studying their reduction modulo $p$. PEL-type
Shimura varieties are moduli space of abelian varieties with manifestly
algebraic conditions (i.e. fixing the data of a polarization, endomorphisms,
and level.) Using the moduli interpretation it is straightforward
to define PEL-type Shimura varieties directly over finite fields $\mathbb{F}_{q}$,
at least for most $q$. As far as we know, there is not as-of-yet
a direct definition of general non-PEL-type Shimura varieties over
$\mathbb{F}_{q}$.

Jie Xia has recently taken the simplest example of non-PEL-type Shimura
curves, what he calls Mumford curves, and given ``direct definitions''
over $\overline{\mathbb{F}}_{p}$ \cite{xia2013crystalline,xia2013tensor,xia2014adic}.
The most basic example of Mumford curves parameterize abelian 4-folds
with certain extra Hodge classes, as in Mumford's original paper \cite{mumford1969note}.
Xia proved theorems of the following form: given an abelian scheme
$\mathcal{A}\rightarrow X$ over a curve $X$, there are certain conditions
that ensure that the pair $(\mathcal{A},X)$ is the reduction of a
Mumford curve together with its universal abelian scheme over $W(\overline{\mathbb{F}}_{p})$.

In Chapter 2 of my thesis \cite{krishnamoorthy2016dynamics}, we posed
the question of characterizing Shimura curves over $\mathbb{F}_{q}$.
Unlike Xia, we \emph{did not} assume the existence of an abelian scheme
$\mathcal{A}$ over the curves considered. Instead, we took as our
starting point Mochizuki's Theorem, which is of group theoretic nature.
\begin{defn*}
Let $X\leftarrow Z\rightarrow Y$ be a correspondence of curves over
$k$. Then we have an inclusion diagram $k(X)\subset k(Z)\supset k(Y)$
of function fields. We say that the correspondence \emph{has no core
}if $k(X)\cap k(Y)$ has transcendence degree 0 over $k$.
\end{defn*}
This definition formalizes the phrase \textquotedbl{}generically unbounded
dynamics\textquotedbl{} (Remark \ref{Remark:Bounded_orbit} and Proposition
\ref{Proposition:no_core_infinite_graph}.) Shimura curves have many
étale correspondences without a core. Inspired by Mochizuki's theorem,
we wondered if all étale correspondences of curves without a core
are \textquotedbl{}related to\textquotedbl{} Shimura curves. Given
a smooth projective curve $X$ over $\mathbb{F}_{q}$, are there other
group theoretic conditions on $\pi_{1}^{\acute{e}t}(X)$ that ensure
that $X$ is the reduction modulo $p$ of a classical Shimura curve?

In this article, we explore the formal structure of (étale) correspondences
without a core with an aim to understanding the similarities with
Hecke correspondences of Shimura curves. We now state the main constructions/results.

Given a correspondence without a core, in Section \ref{Section:Generic_graph}
we construct a pair $(\mathcal{G}_{gen},A)$ of an infinite graph
together with a large topological group of ``algebraic'' automorphisms.
The graph $\mathcal{G}_{gen}$ roughly measures the ``generic dynamics.''
In the case of a symmetric $l$-adic Hecke correspondence of modular
curves, $\mathcal{G}_{gen}$ is a tree and the pair $(\mathcal{G}_{gen},A)$
is related to the action of $PSL_{2}(\mathbb{Q}_{l})$ on its building.
When $\mathcal{G}_{gen}$ is a tree, we prove that the vertices of
$\mathcal{G}_{gen}$ are in bijective correspondence with the maximal
open compact subgroups of a certain subgroup $A^{PQ}$ of $A$ (Corollary
\ref{corollary:max_compact}.) This perhaps suggests that in this
case the action of the topological group $A^{PQ}$ on $\mathcal{G}_{gen}$
is similar to the action of the $l$-adic linear group $PSL_{2}(\mathbb{Q}_{l})$
on its building.
\begin{defn*}
Let $X\overset{f}{\leftarrow}Z\overset{g}{\rightarrow}Y$ be a correspondence
of curves over $k$. A \emph{clump }is a finite set $S\subset Z(\overline{k})$
such that $f^{-1}(f(S))=g^{-1}(g(S))=S$. A clump is \emph{étale}
if $f$ and $g$ are étale at all points of $S$.
\end{defn*}
A clump may be thought of as a \textquotedbl{}bounded orbit of geometric
points.\textquotedbl{} Hecke correspondences of modular curves over
$\mathbb{F}_{p}$ have a natural étale clump: the supersingular locus.
\begin{thm*}
(see Theorem \ref{Theorem:one_clump}) Let $X\overset{f}{\leftarrow}Z\overset{g}{\rightarrow}Y$
be a correspondence of curves over a field $k$ without a core. There
is at most one étale clump.
\end{thm*}
An example: let $l\neq p$. Applying Theorem \ref{Theorem:one_clump}
to the Hecke correspondence $Y(1)\leftarrow Y_{0}(l)\rightarrow Y(1)$
reproves the fact that any two supersingular elliptic curves over
$\overline{\mathbb{F}}_{p}$ are related by an $l$-primary isogeny.
Theorem \ref{Theorem:one_clump} implies a generalization of a theorem
of Hallouin and Perret \cite{hallouin2014recursive}, who came upon
it in the analysis of optimal towers in the sense of Drinfeld-Vladut.
They use spectral graph theory and an analysis of the singularities
of a certain recursive tower. In our language, the hypotheses of their
``one clump theorem'' are
\begin{itemize}
\item $k\cong\mathbb{F}_{q}$
\item $X\leftarrow\Gamma\rightarrow X$ is a minimal self-correspondence
of type $(d,d)$
\item $\mathcal{H}_{gen}$, a certain directed graph where the in-degree
and out-degree of every vertex is $d$, has no directed cycles.
\end{itemize}
Our techniques allow one to relax the third condition to ``$\mathcal{H}_{gen}$
is infinite''; in particular, $\mathcal{H}_{gen}$ may have directed
cycles. Moreover, our proof works over any field $k$ and is purely
algebro-geometric. See the lengthy Remark \ref{Remark:generalize_Perret_Hallouin}
for a full translation/derivation. 

Specializing to characteristic 0, we prove the following.
\begin{thm*}
(see Corollary \ref{Corollary:char_0_no_clumps}) Let $X\leftarrow Z\rightarrow Y$
be an étale correspondence of projective curves over $k$ without
a core. Suppose $\text{char}(k)=0$. Then there are no clumps.
\end{thm*}
We unfurl this statement. Think of a symmetric Hecke correspondence
$X\leftarrow Z\rightarrow X$ of Shimura curves over $\mathbb{C}$
as a many-valued function from $X$ to $X$. Then the iterated orbit
of any point $x\in X$ under this many-valued function is unbounded.
This was likely already known, but we couldn't find it in the literature.
We nonetheless believe our approach is new. We now briefly describe
the sections.

In Section \ref{Section:Correspondences_and_cores} we state Mochizuki's
Theorem (Theorem \ref{Theorem:Mochizuki}). We then reprise the theme:
\textquotedbl{}are all étale correspondences without a core related
to Hecke correspondences of Shimura varieties?\textquotedbl{} in Question
\ref{Question:No_core_Shimura}. The phrase ``related to'' is absolutely
vital, and étale correspondences without a core do not always directly
deform to characteristic 0. We will see one example in Remark \ref{Remark:Igusa}
via Igusa level structure. More exotic is Example \ref{Example:Central_Leaf}
of a central leaf in a Hilbert modular variety; according to a general
philosophy of Chai-Oort, these leaves should also be considered Shimura
varieties in characteristic $p$. Unlike in characteristic 0, however,
these may deform in families purely in characteristic $p$. There
are also examples of étale correspondences of curves over $\overline{\mathbb{F}}_{p}$
without a core using Drinfeld modular curves. We pose a concrete instantiation
of Question \ref{Question:No_core_Shimura} that doesn't mention Shimura
varieties at all (Question \ref{Question:many_hecke}.)

In Section \ref{Sec:A_Recursive_Tower}, starting from a correspondence
without a core, we use elementary Galois theory to construct an infinite
tower of curves $W_{\infty}$ with \textquotedbl{}function field\textquotedbl{}
$E_{\infty}$. We use this tower to prove that the property of \textquotedbl{}not
having a core\textquotedbl{} for an étale correspondences specializes
in families (Lemma \ref{Lem:hyperbolic_specialization}.)

In section \ref{Section:Generic_graph}, given a correspondence without
a core, we construct the pair $(\mathcal{G}_{gen},A)$ of an infinite
graph together with a large group of ``algebraic'' automorphisms.
The graph $\mathcal{G}_{gen}$ packages the Galois theory of $E_{\infty}$
and reflects the generic dynamics of the correspondence. We are especially
interested in Question \ref{Question:no_core_a_tree?}: given an étale
correspondence without a core, is $\mathcal{G}_{gen}$ a tree? Using
Serre's theory of groups acting on trees \cite{serre1977arbres},
we prove that in this case the action of $A^{PQ}$ on $\mathcal{G}_{gen}$
shares several properties with the action of the $l$-adic linear
group $PSL_{2}(\mathbb{Q}_{l})$ on its building (see Proposition
\ref{Proposition:tree_amalgamated_free_product} and Corollary \ref{corollary:max_compact}.)

In Section \ref{Section:Symmetric_correspondences} we develop some
basic results for symmetric correspondences. We are interested in
the following refinement of Question \ref{Question:no_core_a_tree?}:
given a symmetric étale correspondence without a core, is the pair
$(\mathcal{G}_{gen},A)$ $\infty$-transitive (Question \ref{Question:infty_transitive})?
In the case of a symmetric, type (3,3) correspondence without a core,
we are able to verify this using graph theory due to Tutte (Lemma
\ref{Lemma:free_33}); in particular, in this case $\mathcal{G}_{gen}$
is a tree.

In Section \ref{Section:Specialization_of_Graphs} we construct specialization
maps $\mathcal{G}_{gen}\rightarrow\mathcal{G}_{phys}$. These roughly
specialize the dynamics from the generic point to closed points. When
the original correspondence is étale, the maps $\mathcal{G}_{gen}\rightarrow\mathcal{G}_{phys}$
are covering spaces of graphs (Lemma \ref{Lemma:etale_specialization}.)
Motivated by work of Kohel and Sutherland on \emph{Isogeny Volcanoes}
of elliptic curves \cite{kohel1996endomorphism,sutherland2013isogeny},
we speculate on the behavior and asymptotics of these specialization
maps in Question \ref{Question:Special_orbit}.

The rest of the paper may be read independently. In Section \ref{Section:Invariant_Line_Bundles}
we introduce the notion of an \emph{invariant line bundle }on a correspondence
and prove several results about their spaces of sections on (étale)
correspondences without a core. In characteristic 0 there are no \emph{invariant
pluricanonical differential forms} (Proposition \ref{Corollary:no_invariant_sections_char_0}.)
In characteristic $p$, however, there may be such forms. The existence
of the \emph{Hasse invariant}, a mod-$p$ modular form, is a representative
example of the difference. The key to these results is the introduction
of the group scheme $Pic^{0}(X\leftarrow Z\rightarrow Y)$; when $X\leftarrow Z\rightarrow Y$
does not have a core, we prove that this group scheme is finite (Lemma
\ref{Lemma:No_abelian_subvariety}.) We speculate on the relationship
between \emph{invariant differential forms} and $Pic^{0}(X\leftarrow Z\rightarrow Y)$
in Question \ref{Question:invariant_forms_nonreduced_pic}.

In Section \ref{Section:Clumps_and_bundles} we show that an \emph{étale
clump} gives rise to an invariant line bundle together with a line
of invariant sections. Using the analysis in Section \ref{Section:Invariant_Line_Bundles},
we prove the two sample theorems above and explicate the relationship
between our result and that of Hallouin-Perret. We wonder if every
étale correspondence of projective curves without a core in characteristic
$p$ has a clump, equivalently an invariant pluricanonical differential
form (Question \ref{Question:clump_exist?}.) 

We briefly comment on Chapter 3 of my thesis \cite{krishnamoorthy2016dynamics}
(see also the forthcoming article \cite{krishnamoorthyrank2companion}.)
Let $(X\overset{f}{\leftarrow}Z\overset{g}{\rightarrow}X)$ by a symmetric
type (3,3) étale correspondence without a core over a finite field
$\mathbb{F}_{q}$. Inspired by the formal similarity between the pair
$(\mathcal{G}_{gen},A^{PQ})$ and $(\mathcal{T},PSL_{2}(\mathbb{Q}_{2})),$where
$\mathcal{T}$ is the building of $PGL_{2}(\mathbb{Q}_{2})$ (i.e.
the infinite trivalent tree), we assume that the action of $G_{P}$
on $\mathcal{G}_{gen}$ is isomorphic to the action of $PSL_{2}(\mathbb{Z}_{2})$
on $\mathcal{T}$, a \emph{purely group-theoretic condition}. Call
the associated $\mathbb{Q}_{2}$ local system $\mathscr{L}$. Then,
using 2-adic group theory we prove that $f^{*}\mathscr{L}\cong g^{*}\mathscr{L}$.
Suppose further that all Frobenius traces of $\mathscr{L}$ are in
$\mathbb{Q}$. Using a recent breakthrough in the $p$-adic Langlands
correspondence for curves over a finite field due to Abe \cite{abe2013langlands},
we build the following correspondence.
\begin{thm*}
\label{Theorem:p-div_companion_versal} Let $C$ be a smooth, geometrically
irreducible, complete curve over $\mathbb{F}_{q}$. Suppose $q$ is
a square. There is a natural bijection between the following two sets.
\[
\left\{ \begin{alignedat}{1} & \overline{\mathbb{Q}_{l}}\text{-local systems }\mathscr{L}\text{ on }C\text{ such that}\\
 & \bullet\mathscr{L}\text{ is irreducible of rank 2}\\
 & \bullet\mathscr{L}\text{ has trivial determinant}\\
 & \bullet\text{The Frobenius traces are in }\mathbb{Q}\\
 & \bullet\mathscr{L}\text{ has infinite image,}\\
 & \text{up to isomorphism}
\end{alignedat}
\right\} \longleftrightarrow\left\{ \begin{aligned} & p\text{-divisible groups }\mbox{\ensuremath{\mathscr{G}}}\text{ on }C\text{ }\text{ such that }\\
 & \bullet\mathscr{G}\text{ has height 2 and dimension 1}\\
 & \bullet\mathscr{G}\text{ is generically versally deformed}\\
 & \bullet\mathscr{G}\text{ has all Frobenius traces in }\mathbb{Q}\\
 & \bullet\mathscr{G}\text{ has ordinary and supersingular points,}\\
 & \text{up to isomorphism}
\end{aligned}
\right\} 
\]
such that if $\mathscr{L}$ corresponds to $\mathcal{G}$, then $\mathscr{L}\otimes\mathbb{Q}_{l}(-1/2)$
is compatible with the $F$-isocrystal $\mathbb{D}(\mathscr{G})\otimes\mathbb{Q}$.
\end{thm*}
If $\mathscr{G}$ is \emph{everywhere versally deformed} on $X$,
Xia's work \cite{xia2013deformation} shows that the pair $(X,\mathscr{G})$
may be canonically lifted to characteristic 0. In this case the whole
correspondence is the reduction modulo $p$ of an étale correspondence
of Shimura curves. However, examples coming from Shimura curves with
Igusa level structure show that $\mathscr{G}$ may be generically
versally deformed without being everywhere versally deformed. For
more details, see my thesis \cite{krishnamoorthy2016dynamics}.

\textbf{Acknowledgments.} This work is an extension of Chapter 2 of
my PhD thesis at Columbia University. I am very grateful to Johan
de Jong, my former thesis advisor, for guiding this project and for
countless inspiring discussions. Ching-Li Chai read my thesis very
carefully and provided many illuminating corrections and remarks,
especially Example \ref{Example:Central_Leaf}; I thank him. I also
thank Aaron Bernstein, Ashwin Deopurkar, Remy van Dobben de Bruyn,
Hélène Esnault, Ambrus Pal, and especially Philip Engel for interesting
conversations on the topic of this article. During my time at Freie
Universität Berlin I have been funded by an NSF postdoctoral fellowship,
Grant No. DMS-1605825.

\section{Conventions, Notation, and Terminology}

We explicitly state conventions and notations. These are in full force
unless otherwise stated.
\begin{enumerate}
\item $p$ is a prime number and $q$ is a power of $p$.
\item $\mathbb{F}$ is fixed algebraic closure of $\mathbb{F}_{p}$.
\item A \emph{curve} $C$ over a field $k$ is a geometrically integral
scheme of dimension 1 over $k$. Unless otherwise explicitly stated,
we assume $C\rightarrow\text{Spec}(k)$ is smooth.
\item A morphism of curves $X\rightarrow Y$ over $k$ is a morphism of
$k$-schemes that is non-constant, finite, and generically separable.
\item A smooth curve $C$ over a field $k$ is said to be \emph{hyperbolic}
if $\text{Aut}_{\overline{k}}(C_{\overline{k}})$ is finite.
\item Given a field $k$, $\Omega$ will always be an algebraically closed
field of transcendence degree 1 over $k$.
\item In general, $X$, $Y$, and $Z$ will be a curves over $k$, with
$M=k(Z)$, $L=k(X)$, and $K=k(Y)$ the function fields. We \emph{fix
}a\emph{ }$k$-algebra embedding $PQ:k(Z)\hookrightarrow\Omega$ that
identifies $\Omega$ as an algebraic closure of $k(Z)$.
\end{enumerate}

\section{\label{Section:Correspondences_and_cores}Correspondences and Cores}
\begin{defn}
A smooth curve $X$ over a field $k$ is said to be hyperbolic if
$\text{Aut}_{\overline{k}}(X_{\overline{k}})$ is finite.
\end{defn}
This is equivalent to the usual criterion of $2g-2+r\geq1$ where
$g$ is the geometric genus of the compactification $\overline{X}$
and $r$ is the number of geometric punctures. Over the complex numbers,
this is equivalent to $X$ being uniformized by the upper half plane
$\mathbb{H}$.
\begin{lem}
If $X\rightarrow Y$ is a non-constant morphism of curves over $k$
where $Y$ is hyperbolic, then $X$ is hyperbolic.
\end{lem}
\begin{defn}
A \emph{correspondence of curves over $k$} is a diagram 

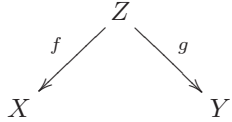
\[
\xymatrix{ & Z\ar[dl]_{f}\ar[dr]^{g}\\
X &  & Y
}
\]
of smooth curves over a field $k$ where $f$ and $g$ are finite,
generically separable morphisms. We call such a correspondence of
\emph{type $(d,e)$} if $\text{deg}f=d$ and $\deg g=e$. We call
such a correspondence \emph{étale} if both maps are étale. We call
a correspondence \emph{minimal} if the associated map $Z\rightarrow X\times Y$
is birational onto its image.
\end{defn}
To a correspondence we can associate a containment diagram of function
fields:
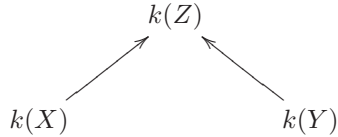
\[
\xymatrix{ & k(Z)\\
k(X)\ar[ur] &  & k(Y)\ar[ul]
}
\]
A correspondence is \emph{minimal} iff there is no proper subfield
of $k(Z)$ that contains both $k(X)$ and $k(Y)$.
\begin{rem}
Note that we require both $f$ and $g$ to be \emph{finite}; for instance,
strict open immersions are not permitted.
\end{rem}
\begin{defn}
\label{Definition:correspondence_having_a_core}We say a correspondence
$X\leftarrow Z\rightarrow Y$ of curves over $k$ has a\emph{ core}
if the intersection of the two function fields $k(X)\cap k(Y)$ has
transcendence degree 1 over $k$.
\end{defn}
\begin{rem}
\label{Remark:core_is_separable}If a correspondence has a core, then
$k(X)\cap k(Y)\subset k(Z)$ is a separable field extension. Indeed,
suppose it weren't. The morphisms $f$ and $g$ are generically separable.
Then at least one of the extensions $k(X)\cap k(Y)\subset k(X)$ or
$k(X)\cap k(Y)\subset k(Y)$ is inseparable. Suppose $k(X)\cap k(Y)\subset k(Y)$
is not separable. Then there exists an element $\lambda\in k(X)$
such that $\lambda\notin k(Y)$ but $\lambda^{p}\in k(Y)$. But $\lambda\in k(Z)$,
so $g$ is not separable, contrary to our original assumption
\end{rem}
Suppose $X\leftarrow Z\rightarrow Y$ is a correspondence of curves
over $k$ with a core. If $X$, $Y$, and $Z$ are projective, we
call the smooth projective curve $C$ associated to the field $k(X)\cap k(Y)$
(considered as a field of transcendence degree 1 over $k$) the \emph{coarse
core} of the correspondence if it exists. One may also define the
coarse core if $X$, $Y$, and $Z$ are affine, see Remark \ref{Remark:affine_core}.

In particular, a correspondence of curves over $k$ has a core if
and only if there exists a curve $C$ over $k$ with finite, generically
separable maps from $X$ and $Y$ such that the following diagram
commutes.
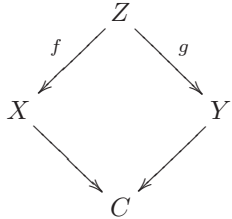
\[
\xymatrix{ & Z\ar[dl]_{f}\ar[dr]^{g}\\
X\ar[dr] &  & Y\ar[dl]\\
 & C
}
\]

\begin{rem}
\label{Remark:Bounded_orbit}Given a correspondence as above, consider
the following ``many-valued function'' $X\dashrightarrow X$ that
sends $x\in X$ to the multi-set $f(g^{-1}(g(f^{-1}(x))))$, i.e.
start with $x$, take all pre-images under $f$, take the image under
$g$, take all pre-images under $g$ and take the image under $f$.
Having a core guarantees that the dynamics of this many-valued function
are uniformly bounded.
\end{rem}
\begin{prop}
\label{Proposition:having_a_core_invariant_under_extension}Let $X\leftarrow Z\rightarrow Y$
be a correspondence of curves over $k$. Let $L$ be an algebraic
field extension of $k$ and $X_{L}\leftarrow Z_{L}\rightarrow Y_{L}$
the base-changed correspondence of curves over $L$. Then $X\leftarrow Z\rightarrow Y$
has a core if and only if $X_{L}\leftarrow Z_{L}\rightarrow Y_{L}$
has a core. 
\end{prop}
\begin{proof}
We may assume $X$, $Y$, and $Z$ are projective. We immediately
reduce to the case of $L/K$ being a finite extension. If $X\leftarrow Z\rightarrow Y$
has a core, then so does $X_{L}\leftarrow Z_{L}\rightarrow Y_{L}$,
so it remains to prove the reverse implication. Let $C$ be the coarse
core of $X_{L}\leftarrow Z_{L}\rightarrow Y_{L}$. Then the universal
property of Weil restriction of scalars applied to $\text{Res}_{L/K}C$
yields the following commutative diagram with non-constant morphisms.
\[
\xymatrix{ & Z\ar[dl]_{f}\ar[dr]^{g}\\
X\ar[dr] &  & Y\ar[dl]\\
 & \text{Res}_{L/K}C
}
\]
Taking the image of $X$ and $Y$ inside of $\text{\text{Res}}_{L/K}C$
allows us to conclude that $X\leftarrow Z\rightarrow Y$ had a core.
\end{proof}
\begin{rem}
In our conventions, a curve $C$ over a field $k$ is geometrically
integral. Therefore $k$ is algebraically closed inside of $k(C)$.
If $X\leftarrow Z\rightarrow Y$ is a correspondence of curves over
$k$ without a core, then the natural map $k\hookrightarrow k(X)\cap k(Y)$
is therefore an isomorphism, as $k(X)\cap k(Y)$ is a finite extension
of $k$ that is contained in $k(X)$.
\end{rem}
For most of this article, we focus on correspondences without a core.
General correspondences of curves will not have cores. Consider, for
instance a correspondence of the form
\[
\xymatrix{ & \mathbb{P}^{1}\ar[dl]_{f}\ar[dr]^{g}\\
\mathbb{P}^{1} &  & \mathbb{P}^{1}
}
\]
For general $f$ and $g$, the dynamics of the induced many-valued-function
$\mathbb{P}^{1}\dashrightarrow\mathbb{P}^{1}$ will be unbounded and
hence it will not have a core by Remark \ref{Remark:Bounded_orbit}.
When we restrict to étale correspondences without cores, there is
the following remarkable theorem of Mochizuki \cite{mochizuki1998correspondences}
(due in large part to Margulis \cite{margulis1991discrete}), which
is the starting point of this article.
\begin{thm}
\label{Theorem:Mochizuki}\cite{mochizuki1998correspondences} If
$X\leftarrow Z\rightarrow Y$ is an étale correspondence of hyperbolic
curves without a core over a field $k$ of characteristic 0 , then
$X,Y,$ and $Z$ are all Shimura (arithmetic) curves (see Definitions
2.2, 2.3 of loc. cit.)
\end{thm}
\begin{rem}
\label{Remark:no_defs}Theorem \ref{Theorem:Mochizuki} in particular
implies that if $X\leftarrow Z\rightarrow Y$ is an étale correspondence
of complex hyperbolic curves, then all of the curves and maps can
be defined over $\overline{\mathbb{Q}}$. In particular, there are
\emph{no non-trivial deformations} of étale correspondences of hyperbolic
curves without a core over $\mathbb{C}$. This last point fails over
$\mathbb{F}$; we will see examples, explained by Ching-Li Chai, later
in Example \ref{Example:Central_Leaf}.
\end{rem}
The proof of Theorem \ref{Theorem:Mochizuki} comes down to a reduction
to $k\cong\mathbb{C}$ by the Lefschetz principle and an explicit
description, due to Margulis \cite{margulis1991discrete}, of the
arithmetic subgroups $\Gamma$ of $SL(2,\mathbb{R})$. Given a complex
hyperbolic curve $C$, fix a uniformization $\mathbb{H}\rightarrow C$
to obtain an embedding
\[
\Gamma:=\pi_{1}(C)\rightarrow SL(2,\mathbb{R})
\]
We say $\gamma\in SL(2,\mathbb{R})$ commensurates $\Gamma$ if the
discrete group $\gamma\Gamma\gamma^{-1}$ is commensurable with $\Gamma,$
i.e. their intersection is of finite index in both groups. Define
$Comm(\Gamma)$ to be the subgroup commensurating $\Gamma$ in $SL(2,\mathbb{R})$
and note that $\Gamma\subset Comm(\Gamma)$. Margulis has proved that
$\Gamma$ is arithmetic if and only if $[Comm(\Gamma):\Gamma]=\infty$,
see e.g. Theorem 2.5 of \cite{mochizuki1998correspondences}.
\begin{example}
The commensurater of $SL(2,\mathbb{Z})$ in $SL(2,\mathbb{R})$ is
$SL(2,\mathbb{Q})$. The modular curve $Y(1)=[\mathbb{H}/SL(2,\mathbb{Z})]$
is arithmetic.
\end{example}
\begin{xca}
\label{Exercise:2_isogeny_hecke_elliptic}The correspondence of non-projective
stacky modular curves $Y(1)\leftarrow Y_{0}(2)\rightarrow Y(1)$ does
not have a core. Here, $Y_{0}(2)$ is the moduli space of pairs of
elliptic curves $(E_{1}\overset{2:1}{\rightarrow}E_{2})$ with a given
degree-2 isogeny between them, and the maps send the isogeny to the
source and target elliptic curve respectively. Hint: to prove this
over the complex numbers, look at the ``orbits'' of $\tau\in\mathbb{H}$
as in Remark \ref{Remark:Bounded_orbit}.
\end{xca}
\begin{rem}
We comment on the definition of an \emph{arithmetic curve.} In Definitions
2.2 and 2.3 of \cite{mochizuki1998correspondences}, Mochizuki defines
two notions of arithmetic (complex) hyperbolic Riemann surface: Margulis
arithmeticity and Shimura arithmeticity. Margulis arithmeticity is
closer in spirit to the classical definition of a Shimura variety,
while Shimura arithmeticity is essentially given by the data of a
totally real field $F$ and a quaternion algebra $D$ over $F$ that
is split at exactly one of the infinite places. Proposition 2.4 then
proves these two definitions are equivalent. If $X$ is an arithmetic
curve and $Y\rightarrow X$ is a finite étale cover, then $Y$ is
manifestly arithmetic by either definition. In particular, the hyperbolic
Riemann surfaces associated to \emph{non-congruence} \emph{subgroups}
of $SL_{2}(\mathbb{Z})$ are arithmetic by definition.
\end{rem}
\begin{defn}
\label{Definition:moduli_fake_elliptic_curves}Let $D$ be an indefinite
non-split quaternion algebra over $\mathbb{Q}$ of discriminant $d$
and let $\mathcal{O}_{D}$ be a fixed maximal order. Let $k$ be a
field whose characteristic is prime to $d$. A \emph{fake elliptic
curve }is a pair $(A,i)$ of an abelian surface $A$ over $k$ and
an injective ring homomorphism $i:\mathcal{O}_{D}\rightarrow\text{End}_{k}(A)$.
The abelian surface $A$ is endowed with the unique principal polarization
such that the Rosati involution induces the canonical involution on
$\mathcal{O}_{D}$. 
\end{defn}
Just as one can construct a modular curve parameterizing elliptic
curves, there is a Shimura curve $X^{D}$ parameterizing fake elliptic
curves with multiplication by $\mathcal{O}_{D}$. Over the complex
numbers, these are compact hyperbolic curves. Explicitly, if one chooses
an isomorphism $D\otimes\mathbb{R}\cong M_{2\times2}(\mathbb{R})$,
look at the image of $\Gamma=\mathcal{O}_{D}^{1}$ of elements of
$\mathcal{O}_{D}^{*}$ of norm 1 (for the standard norm on $\mathcal{O}_{D}$)
inside of $SL(2,\mathbb{R})$. This is a discrete subgroup and in
fact acts properly discontinuously and cocompactly on $\mathbb{H}$.
The quotient $[\mathbb{H}/\Gamma]$ is the Shimura curve associated
to $\mathcal{O}_{D}$. There is a notion of isogeny of fake elliptic
curves which is required to be compatible with the $\mathcal{O}_{D}$
structure and the associated ``fake degree'' of an isogeny. See
Buzzard \cite{buzzard1997integral} or Boutot-Carayol \cite{boutot1991uniformisation}
for more details. These definitions allow us the define Hecke correspondences
as in the elliptic modular case. For instance, as long as 2 splits
in $D$, one can define the correspondence
\[
\xymatrix{ & X_{0}^{D}(2)\ar[dl]_{\pi_{1}}\ar[dr]^{\pi_{2}}\\
X^{D} &  & X^{D}
}
\]
where $X_{0}^{D}(2)$ parametrizes pairs of fake elliptic curves $(A_{1}\rightarrow A_{2})$
with a given isogeny of fake degree 2 between them and $\pi_{1}$
and $\pi_{2}$ are the projections to the source and target respectively.
This is an example of an étale correspondence of (stacky) hyperbolic
curves without a core. To get an example without orbifold points,
one can add auxiliary level structure by picking an open compact subgroup
$K\subset\mathbb{A}^{f}$ of the finite adeles. This correspondence
is in fact defined over $\mathbb{Z}[\frac{1}{2S}]$ for an integer
$S$ and so may be reduced modulo $p$ for almost all primes.

Motivated by these examples, the orienting question of this article
is to explore characteristic $p$ analogs of Mochizuki's theorem.
More specifically, we wish to understand the abstract structure of
étale correspondences of hyperbolic curves without a core.
\begin{question}
\label{Question:No_core_Shimura}Let $k$ be a field of characteristic
$p$. If $X\leftarrow Z\rightarrow Y$ is an étale correspondence
of hyperbolic curves over $k$ without a core, then is it related
to a Hecke correspondence of Shimura varieties or Drinfeld modular
varieties?
\end{question}
\begin{rem}
In Corollary \ref{Corollary:reduce_to_F} we will in some sense reduce
Question \ref{Question:No_core_Shimura} to the analogous question
with $k=\mathbb{F}$.
\end{rem}
The clause ``is related to'' in Question \ref{Question:No_core_Shimura}
is absolutely vital as we will see in the following examples. Nonetheless
we take Question \ref{Question:No_core_Shimura} as a guiding principle. 
\begin{rem}
\label{Remark:Igusa}There are examples of étale correspondence of
hyperbolic curves without a core that should not deform to characteristic
0. Consider, for instance, the Hecke correspondence
\[
\xymatrix{ & Y_{0}(2)\ar[dl]_{\pi_{1}}\ar[dr]^{\pi_{2}}\\
Y(1) &  & Y(1)
}
\]
of modular curves over $\mathbb{F}_{p}$, $p\neq2$. By definition,
there is a universal elliptic curve $\mathcal{E}\rightarrow Y(1)$.
Let $\mathcal{G}=\mathcal{E}[p^{\infty}]$ be the associated $p$-divisible
group over $Y(1)$. Note that $\pi_{1}^{*}\mathcal{G}\cong\pi_{2}^{*}\mathcal{G}$.
Let $X$ be the cover of $Y(1)$ that trivializes the finite flat
group scheme $\mathcal{G}[p]_{\acute{e}t}$ away from the supersingular
locus of $Y(1)$. $X$ is branched exactly at the supersingular points.
Let $Z$ be the analogous cover of $Y_{0}(2)$. Then we have an étale
correspondence
\[
\xymatrix{ & Z\ar[dl]\ar[dr]\\
X &  & X
}
\]
which does not have a core (the dynamics of an ordinary point are
unbounded) and morally one does not expect this correspondence to
lift to characteristic 0. This construction is referred to as adding
\emph{Igusa level structure }in the literature: Ulmer's article \cite{ulmer1990universal}
is a particularly lucid account of this story for modular curves.
See Definition 4.8 of Buzzard \cite{buzzard1997integral} for the
analogous construction for Shimura curves parameterizing fake elliptic
curves. We take up the example of Igusa curves once again in Example
\ref{Example:igusa_form}, from the perspective of the Hasse invariant
and the cyclic cover trick.
\end{rem}
Modular curves with Igusa level structure still parametrize elliptic
curves with some (purely characteristic $p$) level structure. Ching-Li
Chai has provided the following more exotic example which shows that
étale correspondence of hyperbolic curves over a field of characteristic
$p$ may deform purely in characteristic $p$.
\begin{example}
\label{Example:Central_Leaf}Let $F$ be a totally real cubic field
and let $p$ be an inert prime. Consider the Hilbert modular threefold
$\mathcal{X}^{F}$ associated to $\mathcal{O}_{F}$; $\mathcal{X}^{F}$
parametrizes abelian threefolds with multiplication by $\mathcal{O}_{F}$.
Let $X$ be the reduction of $\mathcal{X}^{F}$ modulo $p$ and $\mathscr{A}$
be the universal abelian scheme over $X$. Oort has constructed a
foliation on such Shimura varieties \cite{oort2004foliations}; a
leaf of this foliation has the property that the $p$-divisible group
$\mathscr{A}[p^{\infty}]$ is geometrically constant on the leaf;
i.e., if $x$ and $y$ are two geometric points of the leaf, then
$\mathscr{A}[p^{\infty}]_{x}\cong\mathscr{A}[p^{\infty}]_{y}$. We
list the possible slopes of a height 6, dimension 3, symmetric $p$-divisible
group.
\begin{enumerate}
\item $(0,0,0,1,1,1)$
\item $(0,0,\frac{1}{2},\frac{1}{2},1,1)$
\item $(0,\frac{1}{2},\frac{1}{2},\frac{1}{2},\frac{1}{2},1)$
\item $(\frac{1}{3},\frac{1}{3},\frac{1}{3},\frac{2}{3},\frac{2}{3},\frac{2}{3})$
\item $(\frac{1}{2},\frac{1}{2},\frac{1}{2},\frac{1}{2},\frac{1}{2},\frac{1}{2})$
\end{enumerate}
The only slope types that could possibly admit multiplication (up
to isogeny) by $\mathbb{Q}_{p^{3}}$ are 1, 4, and 5 by considerations
on the endomorphism algebra. By de Jong-Oort purity, the locus where
the slope type $(\frac{1}{3},\frac{1}{3},\frac{1}{3},\frac{2}{3},\frac{2}{3},\frac{2}{3})$
occurs inside of $X$ is codimension 1. One can prove that a central
leaf with this Newton Polygon is a curve. Central leafs of Hilbert
modular varieties have the property that they are preserved under
$l$-adic Hecke correspondences and that in fact the $l$-adic monodromy
is as large as possible \cite{chai2005monodromy}. In particular,
a central leaf has many Hecke correspondences. Moreover, as this Newton
polygon stratum has dimension 2, the isogeny foliation is one-dimensional
and so this Hecke correspondence deforms in a one-parameter family,
purely in characteristic $p$. 
\end{example}
\begin{rem}
Chai and Oort have discussed the possibility that central leaves should
be considered Shimura varieties in characteristic $p$. In particular,
one could consider the example of a Hecke correspondence of a central
leaf to be a Hecke correspondence of Shimura curves. In any case,
both the examples of a Hecke correspondence of modular curves with
Igusa level structure and a Hecke correspondence of a central leaf
of dimension 1 map finitely onto a Hecke correspondence of Shimura
varieties which deform to characteristic 0.
\end{rem}
In my thesis, I phrased Question \ref{Question:No_core_Shimura} only
using Shimura varieties. Ambrus Pal has informed us that there are
examples of étale correspondences of Drinfeld modular curves (i.e.
moduli spaces of $\mathcal{D}$-elliptic modules) over $\mathbb{F}$
without a core. All three of these examples have moduli interpretations.
Moreover, they all have \emph{many} Hecke correspondences. This motivates
the following variant of Question \ref{Question:No_core_Shimura},
which does not mention Shimura/Drinfeld modular varieties at all.
\begin{question}
\label{Question:many_hecke}Let $X\leftarrow Z\rightarrow Y$ be an
étale correspondence of hyperbolic curves over $k$ without a core.
Do there exist infinitely many minimal étale correspondences between
$X$ and $Y$ without a core?
\end{question}

\section{\label{Sec:A_Recursive_Tower}A Recursive Tower}
\begin{defn}
Let $f:X\rightarrow Y$ be a finite, non-constant, generically separable
map of curves over a field $k$. We say $f$ is \emph{finite Galois}
if $|\text{Aut}_{Y}(X)|=\text{deg}(f)$. We say it is \emph{geometrically
finite Galois} if $f_{\overline{k}}:X_{\overline{k}}\rightarrow Y_{\overline{k}}$
is Galois.
\end{defn}
It is well-known that given a finite, generically separable map of
curves over a field $k$, we may take a Galois closure. In the projective
case, this is ``equivalent'' to taking a Galois closure of the associated
extension of function fields, and the affine case follows by the operation
of ``taking integral closure of the coordinate ring in the extension
of function fields.'' However, the output of the ``Galois closure''
operation will not necessarily be a \emph{curve over $k$} as in our
conventions, i.e. it won't necessarily be a geometrically integral
scheme over $k$, unless $k$ is separably closed. For instance, consider
the geometrically Galois morphism $\mathbb{P}_{\mathbb{Q}}^{1}\rightarrow\mathbb{P}_{\mathbb{Q}}^{1}$
given by $t\mapsto t^{3}$. This is not a Galois extension of fields,
and a Galois closure is $\mathbb{P}_{\mathbb{Q}(\zeta_{3})}^{1}$,
which is not a geometrically irreducible variety over $\mathbb{Q}$.
In the language of field theory, the field extension $\mathbb{Q}\subset\mathbb{Q}(\zeta_{3})(t)$
is not \emph{regular}. Therefore, when we take a Galois closure, we
implicitly extended the field $k$ if necessary to ensure that the
output is a \emph{curve over }$k$. 

We begin with a simple Galois-theoretic observation related to the
existence of a core.
\begin{lem}
\label{Lemma:No_Galois}Let $X\leftarrow Z\rightarrow Y$ be a correspondence
over a field $k$ where $Z$ is hyperbolic. A core exists if and only
if there exists a curve $W$, possibly after replacing $k$ by a finite
extension, together with a map $W\rightarrow Z$ such that the composite
maps $W\rightarrow X$ and $W\rightarrow Y$ are both finite Galois.
\end{lem}
\begin{proof}
Suppose such a curve $W$ existed. $W$ is hyperbolic because it maps
nontrivially to a hyperbolic curve. Then the groups $\text{Aut}(W/X)$
and $\text{Aut}(W/Y)$ are both subgroups of $\text{Aut}_{k}(W)$,
which is a finite group because $W$ is hyperbolic. The group $I$
generated by these Galois groups is therefore finite, and the curve
$W/I$ fits into a diagram:

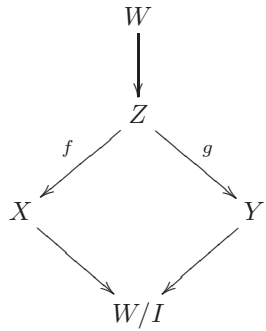
\[
\xymatrix{ & W\ar[d]\\
 & Z\ar[dl]_{f}\ar[dr]^{g}\\
X\ar[dr] &  & Y\ar[dl]\\
 & W/I
}
\]
Therefore a core exists. Conversely, if the correspondence has a core,
call the coarse core $C$. Let $W$ be a Galois closure of the map
$Z\rightarrow C$, finitely extending the ground field if necessary.
Then the composite maps $W\rightarrow X$ and $W\rightarrow Y$ are
both Galois.
\end{proof}
\begin{cor}
\label{Corollary:Finite_No_Galois}Let $X\leftarrow Z\rightarrow Y$
be a correspondence of (possibly non-hyperbolic) curves over $\mathbb{F}$.
Then a core exists if and only if there exists a curve $W$ together
with a map $W\rightarrow Z$ such that the composite maps $W\rightarrow X$
and $W\rightarrow Y$ are both finite Galois.
\end{cor}
\begin{proof}
The proof is almost exactly the same as that of Lemma \ref{Lemma:No_Galois}:
the key observation is that everything in sight (including every element
of $\text{Aut}(W/X)$ and $\text{Aut}(W/Y)$) may be defined over
some finite field $\mathbb{F}_{q}$; therefore the group they generate
inside of $\text{Aut}(W)$ consists of automorphisms defined over
$\mathbb{F}_{q}$ and is hence finite.
\end{proof}
\begin{rem}
\label{Remark:affine_core}Let $X\leftarrow Z\rightarrow Y$ be a
correspondence of affine curves with a core. We prove there exists
a curve $C$ together with \emph{finite, generically separable} maps
from $X$ and $Y$ making the square commute.

Let $\overline{X}\leftarrow\overline{Z}\rightarrow\overline{Y}$ be
the compactified correspondence, with coarse core $T$. Take a Galois
closure $\overline{W}$ of $\overline{Z}\rightarrow T$. Let $W$
be the affine curve associated to the integral closure of $k[Z]$
in $k(\overline{W})$. Then $W\rightarrow X$ and $W\rightarrow Y$
are both finite Galois morphisms of affine curves. In fact, $\text{Aut}(W/X)=\text{Aut}(\overline{W}/\overline{X})$
and likewise for $Y$. The group $I$ generated by $\text{Aut}(W/X)$
and $\text{Aut}(W/Y)$ inside of $\text{Aut}(W)$ is precisely $\text{Aut}(W/T)=\text{Aut}(\overline{W}/T)$,
as $T$ was the coarse core of the projective correspondence. Set
$C=W/I$, the affine curve with coordinate ring $k[W]^{I}$. This
$C$ is the \emph{coarse core} of the correspondence of affine curves.
\end{rem}
\begin{example}
\label{Example:non_hyperbolic_specialization}Let us see the relevance
both of $Z$ being hyperbolic in Lemma \ref{Lemma:No_Galois} and
of the base field being $\mathbb{F}$ in Corollary \ref{Corollary:Finite_No_Galois}.
Let $Z=\mathbb{P}_{\mathbb{F}_{p}(t)}^{1}$ and consider the following
finite subgroups of $PGL(2,\mathbb{F}_{p}(t))$: $H_{1}$ is generated
by the unipotent element $\left(\begin{matrix}1 & t\\
0 & 1
\end{matrix}\right)$ and $H_{2}$ is generated by the unipotent element $\left(\begin{matrix}1 & 0\\
t & 1
\end{matrix}\right)$. Quotienting $Z$ gives a correspondence $Z/H_{1}\leftarrow Z\rightarrow Z/H_{2}$.
Both arrows are Galois, but there is evidently no core because the
subgroup of $PGL(2,\mathbb{F}_{p}(t))$ generated by $H_{1}$ and
$H_{2}$ is infinite. Note that for every specialization of $t\in\mathbb{F}$,
the correspondence does in fact have a core, for instance by Corollary
\ref{Corollary:Finite_No_Galois}.
\end{example}
Let $X\leftarrow Z\rightarrow Y$ be a correspondence of curves without
a core where $Z$ is hyperbolic or where $k\cong\mathbb{F}$. We perform
the following iterative procedure: take a Galois closure of $Z\rightarrow Y$
and call it $W_{Y}$. Because we assumed a core does not exist, the
associated map $W_{Y}\rightarrow X$ cannot be Galois by Lemma \ref{Lemma:No_Galois}
(resp. Corollary \ref{Corollary:Finite_No_Galois}). Take a Galois
closure of this map and call it $W_{YX}$. Again, the associated map
$W_{YX}\rightarrow Y$ cannot be Galois, so we can take a Galois closure
to obtain $W_{YXY}$. Continuing in the fashion, we get an inverse
system of curves $W_{YX\dots}$ over the correspondence. 
\begin{equation}
\xymatrix{ & W_{YX\dots}\ar[d]\\
W_{YXY}\ar[drr] & \vdots\\
 &  & W_{YX}\ar[dll]\\
W_{Y}\ar[dr]\\
 & Z\ar[dl]_{f}\ar[dr]^{g}\\
X &  & Y
}
\label{Tower}
\end{equation}
Note that $W_{YX\dots}$ is Galois over $Z$. In fact, $W_{YX\dots}$
is Galois over both $X$ and $Y$ because there is a final system
of Galois subcovers for each. Note that this procedure may involve
algebraic extensions of the field $k$.

We explicate the based function-field perspective on this construction:
let 
\[
\xymatrix{ & M\\
L\ar[ur] &  & K\ar[ul]
}
\]
be the associated diagram of function fields, where $M=k(Z)$, $L=k(X)$,
and $K=k(Y)$. Recall that $k\overset{\sim}{\rightarrow}L\cap K$;
this is exactly the condition that correspondence does not have a
core.

Pick an algebraic closure $\Omega$ of $M$, i.e. let $\Omega$ be
an algebraically closed field of transcendence degree 1 over $k$
and pick \emph{once and for all} an embedding of $k$-algebras $PQ:M\hookrightarrow\Omega$.
(The notation will be justified later, when $PQ$ will correspond
to an edge of a graph.) Let $E_{K}$ be the Galois closure of $M/K$
in $\Omega$. Then $E_{K}/L$ is no longer Galois by Lemma \ref{Lemma:No_Galois}
(resp. Corollary \ref{Corollary:Finite_No_Galois}). Let $E_{KL}$
be the Galois closure of $E_{K}/L$ in $\Omega$. Continuing in this
fashion, we get an infinite algebraic field extension $E_{KL\dots}$
of $M$, Galois over both $L$ and $K$. 
\begin{lem}
\label{Lemma:final_towers}$W_{YX\dots}$ is isomorphic to $W_{XY\dots}$
as $Z$-schemes. That is, by reversing the roles of $X$ and $Y$
we get mutually final systems of Galois covers.
\end{lem}
\begin{proof}
Equivalently, we must show that $E_{KL\dots}=E_{LK\dots}$ as subfields
of $\Omega$. First, note that $E_{K}\subset E_{LK}$ because $E_{K}$
is the minimal extension of $E$ in $\Omega$ that is Galois over
$K$. Similarly, $E_{KL}\subset E_{LKL}$ because $E_{KL}$ is the
minimal extension of $E_{K}$ in $\Omega$ that is Galois over $L$.
Continuing, we see that $E_{KL\dots}\subset E_{LK\dots}$. By symmetry,
the reverse inclusion holds as desired.
\end{proof}
\begin{cor}
\label{Corollary:Intrinsic_min_galois_both}The field extension $E_{KL\dots}=E_{LK\dots}$
of $M$, thought of as a subfield of $\Omega$, is characterized by
the property that it is the minimal field extension of $M$ inside
of $\Omega$ that is Galois over both $L$ and $K$. 
\end{cor}
For brevity, we denote the inverse system $W_{XYX\dots}$ by $W_{\infty}$.
Let $E_{\infty}$ be the associated function field, considered as
a subfield of $\Omega$. In what follows, unless otherwise specified
we consider $E_{\infty}\subset\Omega$ as inclusions of abstract $k$-algebras.
\begin{question}
\label{Question:Field_of_Constants}Does the ``field of constants''
of $E_{\infty}$ have finite degree over $k$? That is, does $E_{\infty}\otimes_{k}\overline{k}$
decompose as an algebra to be the product of finitely many fields?
What if the original correspondence is étale?
\end{question}
If $\text{Gal}(k^{sep}/k)$ is abelian and $\text{Gal}(E_{\infty}/K)$
has finite abelianization, then Question \ref{Question:Field_of_Constants}
has an affirmative answer. In particular, this applies if $k\cong\mathbb{F}_{q}$
and $\text{Gal}(E_{\infty}/K)$ is a semi-simple $l$-adic group.
We will see in Proposition \ref{Proposition:if_clump_then_bounded_field_of_constants}
that if the correspondence has an \emph{étale clump} , then Question
\ref{Question:Field_of_Constants} has an affirmative answer using
the following remark.
\begin{rem}
\label{Remark:Many maps from W_Y to Z}In Diagram \ref{Tower}, the
morphism $W_{Y}\rightarrow Z$ is Galois. By precomposing with $\text{Aut}(W_{Y}/Z)$,
we equip $W_{Y}$ with $\text{Aut}(W_{Y}/Z)$-many maps to $Z$. More
generally, all curves $W_{YX\dots X}$ will be naturally equipped
with $\text{Aut}(W_{YX\dots X}/Z)$-many maps to $Z$ via precompositions
by Galois automorphisms. This will be useful in Remark \ref{Remark:W_infty_and_G_phys},
when we try to explicitly understand the curves $W_{YX\dots Y}$.
\end{rem}
The following lemma allows us to \emph{specialize }étale correspondences
without a core.
\begin{lem}
\label{Lem:hyperbolic_specialization}Let $S$ be an irreducible scheme
of finite type with generic point $\eta$. Let $X$, $Y$, and $Z$
be proper, smooth, geometrically integral curves over $S$. Suppose
$Z$ is ``hyperbolic'' over $S$; that is, the genus of a fiber
is at least 2. Let $X\leftarrow Z\rightarrow Y$ be an finite étale
correspondence of schemes commuting with the structure maps to $S$.
If $s$ is a geometric point of $S$ such that
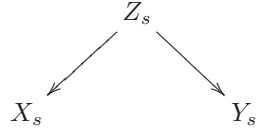
\[
\xymatrix{ & Z_{s}\ar[dl]\ar[dr]\\
X_{s} &  & Y_{s}
}
\]
has a core, then $X_{\eta}\leftarrow Z_{\eta}\rightarrow Y_{\eta}$
has a core.
\end{lem}
\begin{proof}
The property of \textquotedbl{}having a core\textquotedbl{} does not
change under algebraic field extension by Proposition \ref{Proposition:having_a_core_invariant_under_extension}.
By dévissage, we reduce to the case of $S=\text{Spec}(R)$, where
$R$ is a discrete valuation ring with algebraically closed residue
field $\kappa$. Call the fraction field $K$. We may further replace
$R$ by its integral closure in $\overline{K}$ to get a valuation
ring having both the residue field and the fraction field algebraically
closed. We do this to not worry about the \textquotedbl{}extension
of ground field\textquotedbl{} question that is always present when
taking a Galois closure.

Call the generic point $\eta$ and the closed point $s$. First of
all $X_{\eta}\leftarrow Z_{\eta}\rightarrow Y_{\eta}$ is a correspondence
of curves over $\eta$. Let us suppose it does not have a core. Then
the process of iterated Galois closure, as detailed in Diagram \ref{Tower},
continues endlessly to produce a tower of curves over $\eta$. On
the other hand, any finite étale morphism has a Galois closure. This
implies that we can apply the construction of taking iterated Galois
closures to the \emph{finite étale correspondence of schemes} $X\leftarrow Z\rightarrow Y$
to build a tower $W_{YX\dots}$ over $S$. 

As $R$ has algebraically closed residue field and fraction field,
$W_{YX\dots X}$ is a smooth proper curve over $S$; in particular
the geometric fibers of the morphism $W_{YX\dots X}\rightarrow S$
are irreducible. Moreover, all of the maps $W_{YX\dots Y}\rightarrow Z$
are finite étale. In fact, the fiber of $W_{XY\dots X}$ over the
generic point $\eta$ of $W_{XY\dots X}$ is isomorphic, as a scheme
over $Z_{\eta}$, to corresponding curve in the tower associated to
the correspondence $X_{\eta}\leftarrow Z_{\eta}\rightarrow Y_{\eta}$
of curves over $\eta$. For instance, $(W_{Y})_{\eta}\rightarrow Y_{\eta}$
is a Galois closure of the finite étale morphism $Z_{\eta}\rightarrow Y_{\eta}$.
Therefore, if we could prove $(W_{YX\dots X})_{\eta}$ were disconnected,
we would get a contradiction with the original assumption that $X_{\eta}\leftarrow Z_{\eta}\rightarrow Y_{\eta}$
had no core.

The fact that the maps $Z\rightarrow X$, $Z\rightarrow Y$, and $W_{YX\dots Y}\rightarrow Z$
are finite étale implies that taking a Galois closure and then restricting
to $s$ yields a finite Galois étale cover of $Z_{s}$. For example,
$(W_{Y})_{s}$ is a (possibly disconnected) finite Galois cover of
$Y_{s}$ that maps surjectively to $(W_{s})_{Y_{s}}$, a Galois closure
of the map $Z_{s}\rightarrow Y_{s}$.

As the correspondence specialized to $s$ has a core, Lemma \ref{Lemma:No_Galois}
implies that there exists a curve $W_{YX\dots Y}$ of our tower over
$S$ such that the fiber $(W_{YX\dots Y})_{s}$ is disconnected. We
therefore have a smooth proper curve $W_{XY\dots Y}\rightarrow S$
such that the fiber over $s$ is disconnected. Zariski's connectedness
principle implies $(W_{YX\dots Y})_{\eta}$ is disconnected (this
is where we use properness), contradicting our original assumption
that $X_{\eta}\leftarrow Z_{\eta}\rightarrow Y_{\eta}$ had no core.
\end{proof}
\begin{rem}
Example \ref{Example:non_hyperbolic_specialization} shows that the
argument of Lemma \ref{Lem:hyperbolic_specialization} \emph{does
not work }if the correspondence is not assumed to be étale.
\end{rem}
\begin{cor}
\label{Corollary:reduce_to_F}We may ``reduce'' the study of Question
\ref{Question:No_core_Shimura} to where $k=\mathbb{F}$. That is,
given an étale correspondence of hyperbolic curves $X\leftarrow Z\rightarrow Y$
without a core over a field $k$ of characteristic $p$, we can specialize
to an étale correspondence without a core over $\mathbb{F}$. 
\end{cor}
\begin{proof}
By spreading out, we may ensure that we are in the situation of Lemma
\ref{Lem:hyperbolic_specialization}. Then the nonexistence of a core
implies the same for all of the geometric fibers by Lemma \ref{Lem:hyperbolic_specialization}.
\end{proof}
Lemma \ref{Lem:hyperbolic_specialization} says that for an étale
correspondence of projective hyperbolic curves, the property of ``not
having a core'' specializes. The converse is true in more generality
and is rather useful: it implies that one way to answer Question \ref{Question:No_core_Shimura}
is to directly lift the correspondence to characteristic 0.
\begin{lem}
\label{Lemma:generic_has_core_then_special_has_core}Let $S=\text{Spec}(R)$
be the spectrum of a discrete valuation ring with closed point $s$
and generic point $\eta$. Let $X$, $Y$, and $Z$ be smooth, projective,
geometrically irreducible curves over $S$ and let $X\leftarrow Z\rightarrow Y$
be a correspondence of schemes, commuting with the structure maps
to $S$, that is a correspondence of curves when restricted to $s$
and to $\eta$. If over $\eta$ the correspondence has a core, then
over $s$ the correspondence has a core.
\end{lem}
\begin{proof}
Let $\pi$ be a uniformizer of $R$. Denote by $\kappa$ residue field
of $R$ and by $K$ the fraction field of $R$. Pick a non-constant
rational function $f$ in the intersection $K(X)\cap K(Y)$ (the intersection
takes place in $K(Z)$.) By multiplying by an appropriate power of
$\pi$, we can guarantee that $f$ extends to rational functions on
the special fiber and in fact that $f$ has nonzero reduction in $0\neq\overline{f}\in\kappa(X_{s})\cap\kappa(Y_{s})$.
Suppose $f$ is constant modulo $\pi$, or equivalently that $f\equiv c(\text{mod}\pi)$
for some $c\in R$. Then $\frac{f-c}{\pi}$ may again be reduced modulo
$\pi$. If $\frac{f-c}{\pi}$ is non-constant on the special fiber,
we are done, so suppose not and repeat the procedure. This procedure
terminates because our original choice of $f\in K(X)$ was non-constant
and the result will be a non-constant function in $\kappa(X_{s})\cap\kappa(Y_{s})$.
\end{proof}
\begin{cor}
\label{Corollary:Existence_of_a_lift_implies_shimura_curves}Let $X\leftarrow Z\rightarrow Y$
be an étale correspondence of projective hyperbolic curves without
a core over $\mathbb{F}$. If this correspondence lifts to a correspondence
of curves $\tilde{X}\leftarrow\tilde{Z}\rightarrow\tilde{Y}$ over
$W(\mathbb{F})$, then $X$, $Y$, and $Z$ are the reductions modulo
$p$ of Shimura curves.
\end{cor}
\begin{proof}
The lifted correspondence is automatically étale by purity. Lemma
\ref{Lemma:generic_has_core_then_special_has_core} then implies that
the general fiber does not have a core. Mochizuki's Theorem \ref{Theorem:Mochizuki}
then implies that $\tilde{X}$, $\tilde{Y}$, and $\tilde{Z}$ are
all Shimura curves as desired.
\end{proof}

\section{\label{Section:Generic_graph}The Generic Graph of a Correspondence}

Let $X\leftarrow Z\rightarrow Y$ be a correspondence of curves over
$k$. and let $\Omega$ be an algebraically closed field of transcendence
degree 1 over $k$, thought of as a $k$-algebra. We construct an
infinite 2-colored graph $\mathcal{G}_{gen}^{full}$, which we call
the \emph{full generic graph} of the correspondence. The blue vertices
of $\mathcal{G}_{gen}^{full}$ are the $\Omega$-valued points of
$X$; more precisely, a blue vertex is given by a $k$-algebra homomorphism
$k(X)\rightarrow\Omega$. Similarly, the red vertices are the $\Omega$-valued
points of $Y$ and the edges are the $\Omega$-valued points of $Z$.
A blue vertex $p:k(X)\hookrightarrow\Omega$ and red vertex $q:k(Y)\hookrightarrow\Omega$
are joined by an edge if there exists an embedding $k(Z)\hookrightarrow\Omega$
that restricts to $p$ and to $q$ on the subfields $k(X)$ and $k(Y)$
respectively. Note that $\text{Aut}_{k}(\Omega)$ naturally acts on
the graph $\mathcal{G}_{gen}^{full}$ by post-composition.
\begin{rem}
The original correspondence is minimal iff there are no multiple edges
in $\mathcal{G}_{gen}^{full}$. (Recall that the morphisms of curves
were generically separable by definition.)
\end{rem}
\begin{condition}
For the rest of the sections involving graph theory, we suppose that
the correspondence $X\leftarrow Z\rightarrow Y$ is minimal in order
that we get a graph and not a multigraph.
\end{condition}
\begin{defn}
\label{Definition:subfield_subgraph}Given any subgraph $H\subset\mathcal{G}_{gen}^{full}$,
we define the subfield $E_{H}\subset\Omega$ by taking the compositum
of the subfields $e(k(Z))\subset\Omega$, $p(k(X))\subset\Omega$,
and $q(k(Y))\subset\Omega$ corresponding to all of the edges and
vertices $e$, $p$, and $q$ in $H$.
\end{defn}
There is no reason to believe that $\mathcal{G}_{gen}^{full}$ is
connected. We give $\mathcal{G}_{gen}^{full}$ a distinguished blue
vertex $P$, red vertex $Q$, and edge $PQ$ between them by picking
the $k$-embedding 
\[
PQ:k(Z)\hookrightarrow\Omega
\]
and we set the graph $\mathcal{G}_{gen}$ (the \emph{generic graph})
to be the connected component of $\mathcal{G}_{gen}^{full}$ containing
this distinguished edge. All connected components of $\mathcal{G}_{gen}^{full}$
arise in this way and all connected components of $\mathcal{G}_{gen}^{full}$
are isomorphic. We denote by $P(k(X))$ the image of the distinguished
blue point $P$ as a subfield of $\Omega$ and similarly for $Q(k(Y))$.
\begin{lem}
\label{Lemma:Galois_closure_closed_ball}Let $H\subset\mathcal{G}_{gen}$
be the full subgraph consisting of all vertices of distance at most
$n$ from a fixed vertex $v$; that is, $H$ is the closed ball $H=B(v,n)$.
Then $E_{H}$ is Galois over $E_{v}$.
\end{lem}
\begin{proof}
First of all, $E_{v}$ is the field corresponding to $v$ as in Definition
\ref{Definition:subfield_subgraph}. We may suppose WLOG that $v$
is a blue vertex, so $E_{v}=v(k(X))$ as $v$ is by definition a $k$-embedding
$k(X)$ to $\Omega$. In other words, $v$ gives $\Omega$ the structure
of a $k(X)$-algebra. Now, $E_{H}$ is the compositum of all of the
fields associated to all of the edges and vertices in $H$ in $\Omega$.
In particular, if $\mathcal{P}=\{P\}$ is the collection of all paths
of length $n$ starting at $v$, then $E_{H}$ is the compositum of
$(E_{P})_{P\in\mathcal{P}}$ inside of $\Omega$. Here each $E_{P}$
and $E_{H}$ has a $k(X)$-algebra structure via $v$ and our goal
is to prove that $E_{H}$ is Galois over $k(X)$ with respect to this
algebra structure $v:k(X)\hookrightarrow E_{H}$.

Consider $E_{H}$ together with the subfields $E_{P}$, $P\in\mathcal{P}$,
as abstract $k(X)$-algebras. Let $\phi_{0}$ be the original embedding
$E_{H}\hookrightarrow\Omega$. To prove $E_{H}$ is Galois over $k(X)$,
we must show that for every 
\[
\phi\in\text{Hom}_{k(X)}(E_{H},\Omega)
\]
the image of $\phi$ is contained in $\phi_{0}(E_{H})$. Note that
$\phi$ is determined by where all of the $E_{P}$ are sent. Any $\phi$
can be obtained from $\phi_{0}$ via an element of $\text{Aut}(\Omega/k(X))$,
as $\Omega$ is algebraically closed, and so a path $P$ of length
$n$ originating at $v$ is sent to another such path $P'$. In other
words, $\phi(E_{P})=\phi_{0}(E_{P'})$ for another path $P'$ of length
$n$ originating at $v$. As $E_{H}$ was the compositum of all such
$E_{P}$, it follows that the extension $E_{H}/k(X)$ is Galois as
desired. 
\end{proof}
The graph $\mathcal{G}_{gen}$ is a full subgraph of $\mathcal{G}_{gen}^{full}$
so, as in Definition \ref{Definition:subfield_subgraph}, we can take
the associated field $E_{\mathcal{G}_{gen}}\subset\Omega$ given by
the compositum of the subfields of $\Omega$ associated to the edges.
Let $E\subset\Omega$ be the minimal field extension of $k(Z)$ (with
respect to the embedding $PQ:k(Z)\hookrightarrow\Omega$) that is
Galois over both $k(X)$ and $k(Y)$. We prove that $E=E_{\mathcal{G}_{gen}}$
with the next series of results.
\begin{cor}
\label{Corollary:E_infty_in_E_gen}The subfield $E_{\mathcal{G}_{gen}}\subset\Omega$
is Galois over both $P(k(X))$ and $Q(k(Y))$. Therefore $E\subset E_{\mathcal{G}_{gen}}$
\end{cor}
\begin{proof}
The connected graph $\mathcal{G}_{gen}$ is the union of the subgraphs
$\cup_{n}B(P,n)$ of closed balls of radius $n$ around $P$, so by
Lemma \ref{Lemma:Galois_closure_closed_ball} the field $E_{\mathcal{G}_{gen}}$
is Galois over $P(k(X))$. Similarly, $E_{\mathcal{G}_{gen}}$ is
Galois over $Q(k(Y))$. Therefore $E\subset E_{\mathcal{G}_{gen}}$
as desired.
\end{proof}
\begin{lem}
\label{Lemma:E_G_in_biGalois_extension}Let $X\leftarrow Z\rightarrow Y$
be a correspondence of curves over $k$ and embed the function fields
into $\Omega$ via $PQ:k(Z)\hookrightarrow\Omega$. If there is a
subfield $F\subset\Omega$ that is Galois over both $k(X)$ and $k(Y)$,
then $E_{\mathcal{G}_{gen}}\subset F$. 
\end{lem}
\begin{proof}
We have the following diagram of fields
\[
\xymatrix{ & F\\
 & k(Z)\ar[u]\\
k(X)\ar[ur]^{f^{*}} &  & k(Y)\ar[ul]_{g^{*}}
}
\]
where $F$ is Galois over \emph{both }$k(X)$ and $k(Y)$. The field
$F$ is naturally equipped with the structure of a $k(Z)$ algebra.
Extend $PQ:k(Z)\hookrightarrow\Omega$ any which way to a $k(Z)$-algebra
embedding $\phi:F\rightarrow\Omega$. Then the image of any edge adjacent
to $P$ in $\mathcal{G}_{gen}$ lands inside of the image $\phi(F)$
because $F$ is Galois over $k(X)$. Similarly, the image of any edge
adjacent to $Q$ in $\mathcal{G}_{gen}$ lives inside the image of
$\phi(F)$.

Let $q\neq Q$ be a vertex adjacent to $P$. There exists an automorphism
$\alpha\in\text{Gal}(\phi(F)/P(k(X)))$ that sends $Q(k(Y))$ to $E_{q}$
because $F$ is Galois over $k(X)$. Conjugating by $\alpha$, we
deduce that $\phi(F)$ is Galois over $E_{q}$ and hence the image
of all edges emanating from $q$ lie in $\phi(F)$. By propagating,
we get that $E_{\mathcal{G}_{gen}}\subset F$ as desired.
\end{proof}
\begin{cor}
We have an equality of fields $E=E_{\mathcal{G}_{gen}}$, considered
as subfields of $\Omega$. Equivalently, $E_{\mathcal{G}_{gen}}$
is the minimal field extension of $PQ(k(Z))$ inside of $\Omega$
that is Galois over the fields $P(k(X))$ and $Q(k(Y))$. 
\end{cor}
\begin{proof}
Combine Lemma \ref{Lemma:E_G_in_biGalois_extension} and Corollary
\ref{Corollary:E_infty_in_E_gen}.
\end{proof}
\begin{cor}
\label{Corollary:two_defs_e_infinity}Let $X\leftarrow Z\rightarrow Y$
be a correspondence of curves over $k$ without a core with $Z$ hyperbolic
or with $k\cong\mathbb{F}$. Then $E_{\infty}\cong E_{\mathcal{G}_{gen}}$.
\end{cor}
\begin{proof}
The field $E_{\infty}$ is also the minimal field extension of $PQ(k(Z))$
inside of $\Omega$ that is Galois over $P(k(X))$ and $Q(k(Y))$
by Corollary \ref{Corollary:Intrinsic_min_galois_both}.
\end{proof}
\begin{rem}
One is tempted to make a converse definition to Definition \ref{Definition:subfield_subgraph}:
given any subfield $E\subset\Omega$ (respectively $E\subset E_{\infty}$),
define $\mathcal{G}_{E}^{full}$ (respectively $\mathcal{G}_{E}$)
to be the subgraph of $\mathcal{G}_{gen}^{full}$ (respectively $\mathcal{G}_{gen}$)
whose points and edges are have image inside of $E$. This definition
is rather poorly behaved; for instance if one starts out with a finite
connected subgraph $H\subset\mathcal{G}_{gen}$, takes $E_{H}\subset E_{\infty}$,
and then looks at the associated graph $G_{E_{H}}$, there is no reason
to believe that this graph is connected.
\end{rem}
The graph $\mathcal{G}_{gen}$ informally reflects the ``generic
dynamics'' of the correspondence. We will see one way of making this
precise in Section \ref{Section:Specialization_of_Graphs}: via a
specialization map. Nevertheless, we have the following proposition,
which says that a core exists iff $\mathcal{G}_{gen}$ is finite (i.e.
the ``generic dynamics'' are bounded), in line with Remark \ref{Remark:Bounded_orbit}.
\begin{prop}
\label{Proposition:no_core_infinite_graph}Let $X\leftarrow Z\rightarrow Y$
be a correspondence of curves over $k$ where $Z$ is hyperbolic or
where $k\cong\mathbb{F}$. This correspondence has no core if and
only if $\mathcal{G}_{gen}$ is an infinite graph.
\end{prop}
\begin{proof}
If $\mathcal{G}_{gen}$ is finite, then $E_{\mathcal{G}_{gen}}$ is
a finite Galois extension of both $k(X)$ and $k(Y)$, so the correspondence
has a core by Lemma \ref{Lemma:No_Galois} (resp. Corollary \ref{Corollary:Finite_No_Galois}.)

Conversely, if the correspondence had a core, then let $C$ be the
coarse core. Let $W$ be a Galois closure of $Z\rightarrow C$. We
have the following diagram of fields, where we again fix $PQ:k(Z)\hookrightarrow\Omega$
and any extension $\phi:k(W)\hookrightarrow\Omega$. 
\[
\xymatrix{ & k(W)\\
 & k(Z)\ar[u]\\
k(X)\ar[ur]^{f^{*}} &  & k(Y)\ar[ul]_{g^{*}}\\
 & k(C)\ar[ul]\ar[ur]
}
\]
Call $P$, $Q$, and $R$ the restriction of the algebra embedding
$PQ$ to $k(X)$, $k(Y)$, and $k(C)$ respectively. Let $v$ be blue
vertex in $\mathcal{G}_{gen}$ adjacent to $Q$, given by a $k$-algebra
embedding $v:k(X)\hookrightarrow\phi(k(W))\subset\Omega$ by Lemma
\ref{Lemma:E_G_in_biGalois_extension}. As $k(W)/k(Y)$ is Galois,
there exists an automorphism 
\[
\alpha\in\text{Gal}(\phi(k(W))/Q(k(Y)))\cong\text{Gal}(k(W)/k(Y))
\]
that sends $P$ to $v$. As $k(C)\subset k(Y)$, this implies that
$v|_{k(C)}=R$. By propagating, we see that for every vertex $v$
of $\mathcal{G}_{gen}$, $v|_{k(C)}=R$. Therefore, for every edge
$e\in\mathcal{G}_{gen}$, thought of as a $k$-algebra embedding $e:k(Z)\hookrightarrow\Omega$,
we have that $e|_{k(C)}=R$. On the other hand, $k(Z)$ is a finite
extension of $k(C)$, so there are only finitely many ways to extend
$R$ to a $k$-algebra homomorphism $k(Z)\hookrightarrow\Omega$.
Therefore the number of edges is finite, as desired. 
\end{proof}
We record the following easy proposition for later use in proving
the surjective of the specialization morphism in the case of an étale
correspondence without a core.
\begin{prop}
\label{Proposition:finite_distance}For any finite subgraph $H\in\mathcal{G}_{gen}$,
the field $E_{H}$ is contained in a finite extension $F$ of $PQ(k(Z))$.
\end{prop}
\begin{proof}
We have the the following two facts.

\begin{itemize}
\item $E_{H}$ lands inside of $E_{\infty}$, which is exhausted by fields
of the form $E_{KL\dots K}$, by Lemma \ref{Lemma:E_G_in_biGalois_extension}
\item $E_{H}$ is finitely generated over $k$.
\end{itemize}
Therefore $E_{H}$ lands inside of some $E_{KL\dots K}$, a finite
extension of $PQ(k(Z))$, as desired.
\end{proof}
We now analyze the action of various subgroups of $\text{Aut}(E_{\infty})$
on $\mathcal{G}_{gen}$.
\begin{rem}
\label{Remark:Topology_on_automorphism_groups}We take a brief digression
into the structure of automorphism groups of fields. Let $\Omega$
be any field. We endow the group $\text{Aut}(\Omega)$ with the compact-open
topology, considering $\Omega$ to be a discrete set. Given any finite
subset $S\subset\Omega$, the subgroup $\text{Stab}(S)\subset\text{Aut}(\Omega)$
is an open subgroup and as $S$ ranges these form a neighborhood base
of the identity in $\text{Aut}(\Omega)$. If $K\subset\Omega$ is
a separable Galois extension with $K$ finitely generated over its
prime field, the natural map $\text{Gal}(\Omega/K)\subset\text{Aut}(\Omega)$
is an open embedding of topological groups; in other words, the topology
just defined is compatible with the usual profinite topology on Galois
groups.

Note that this procedure generalizes: if $k\subset\Omega$ is a field
extension, we may give the group $\text{Aut}_{k}(\Omega)$ has the
structure of a topological group, where a neighborhood base of the
identity is given by $\text{Stab}(S)$ for finite subsets $S\subset\Omega\backslash k$.
However, $\text{Aut}_{k}(\Omega)$ is not an \emph{open subgroup}
of $\text{Aut}(\Omega)$ unless $k$ is finitely generated over its
prime field.
\end{rem}
Any element $g\in\text{Aut}_{k}(E_{\infty})$ gives a map of graphs
$\mathcal{G}_{gen}\rightarrow\mathcal{G}_{gen}^{full}$ by post-composition:
for instance, an edge $e:k(Z)\rightarrow E_{\infty}\subset\Omega$
is sent to the edge $g\circ e:k(Z)\rightarrow E_{\infty}\subset\Omega$.
In fact, the Galois groups $G_{P}:=\text{Gal}(E_{\infty}/P(k(X)))$
and $G_{Q}:=\text{Gal}(E_{\infty}/Q(k(Y)))$ actually act on the connected
graph: $g\in G_{P}$ sends an edge $e:k(Z)\rightarrow E_{\infty}\subset\Omega$
to $g\circ e:k(Z)\rightarrow E_{\infty}\subset\Omega$, and $g\circ e$
is an edge of the connected graph $\mathcal{G}_{gen}$ because $g$
fixes $P$.
\begin{defn}
\label{Definition:A}Let $A\subset\text{Aut}_{k}(E_{\infty})$ the
subgroup of $\text{Aut}_{k}(E_{\infty})$ sends $\mathcal{G}_{gen}$
to itself with the induced topology, as in Remark \ref{Remark:Topology_on_automorphism_groups}.
Let $A^{PQ}\subset A$ be the subgroup of $A$ generated by $G_{P}$
and $G_{Q}$ with the induced topology from $A$.
\end{defn}
\begin{question}
Is $A\hookrightarrow\text{Aut}_{k}(E_{\infty})$ an isomorphism?
\end{question}
\begin{rem}
The topology on $A^{PQ}$ is uniquely determined by declaring the
compact subgroups $G_{P}$ and $G_{Q}$ to be open.
\end{rem}
By definition, $A$ acts faithfully on $\mathcal{G}_{gen}$: if $g\in A$
acts trivially on $\mathcal{G}_{gen}$, then it acts trivially on
the field generated by all of the vertices and the edges of $\mathcal{G}_{gen}$,
i.e. it is the trivial automorphism of $E_{\infty}$. If we give $\mathcal{G}_{gen}$
the discrete topology, $A^{PQ}$ acts continuously on $\mathcal{G}_{gen}$;
that is, the stabilizer of a vertex is an open subgroup. Let $d=\text{deg}(Z\rightarrow X)$
and $e=\text{deg}(Z\rightarrow Y)$. Then the degree of a blue vertex
is $d$ and the degree of a red vertex is $e$. Moreover, $G_{P}$
acts transitively on the edges coming out of $P$ by Galois theory
and similarly $G_{Q}$ acts transitively on the edges coming out of
$Q$. By conjugating we see that $A^{PQ}\subset\text{Aut}(\mathcal{G}_{gen})$
acts transitively on the edges coming out of any vertex. Therefore
the group $A^{PQ}$ acts transitively on the edges of $\mathcal{G}_{gen}$,
subject to the constraint that colors of the vertices are preserved.
This is recorded in the following corollary.
\begin{cor}
\label{Corollary:Colored_edge_symmetric}In the notation above, $A^{PQ}$
and hence also $A$ act transitively on the edges of $\mathcal{G}_{gen}$,
subject to the constraint that the colors of the vertices are preserved.
We say the pair $(\mathcal{G}_{gen},A^{PQ})$ is colored-edge-symmetric.
\end{cor}
\begin{question}
\label{Question:no_core_a_tree?}If $X\leftarrow Z\rightarrow Y$
is a minimal correspondence with no core, does $\mathcal{G}_{gen}$
have any cycles? What if it is étale?
\end{question}
The graph $\mathcal{G}_{gen}$ being a tree has consequences for the
structure of the group $A^{PQ}$. To state these, we need a theorem
of Serre.
\begin{thm}
(Serre) Let $G$ be a group acting on a graph $X$, and let $e$ be
an edge of $X$ connecting vertices $p$ and $q$. Suppose that $e$
is a fundamental domain for the action. Let $G_{p}$, $G_{q}$, and
$G_{e}$ be the stabilizers in $G$ of $p$, $q$, and $e$ respectively.
Then the following are equivalent.

\begin{enumerate}
\item $X$ is a tree
\item The homomorphism $G_{p}*_{G_{e}}G_{q}\rightarrow G$ induced by the
inclusions $G_{p}\rightarrow G$ and $G_{q}\rightarrow G$ is an isomorphism
\end{enumerate}
\end{thm}
\begin{proof}
This is a direct translation of Théorèm 6 on Page 48 of \cite{serre1977arbres}. 
\end{proof}
\begin{prop}
\label{Proposition:tree_amalgamated_free_product}Suppose $\mathcal{G}_{gen}$
is a tree. Then the natural map $G_{P}*_{G_{PQ}}G_{Q}\rightarrow A^{PQ}$
is an isomorphism of topological groups.
\end{prop}
\begin{proof}
There is no element $a\in A^{PQ}$ that flips any edge $e$ of $\mathcal{G}_{gen}$
because $A^{PQ}$ preserves the coloring. By Corollary \ref{Corollary:Colored_edge_symmetric},
the segment $PQ$ is a fundamental domain for the action of $A^{PQ}$
on $\mathcal{G}_{gen}$. Therefore, by Serre's Theorem, the fact that
$\mathcal{G}_{gen}$ is a tree implies the induced map $G_{P}*_{G_{PQ}}G_{Q}\rightarrow A^{PQ}$
is an isomorphism of abstract groups. The group $G_{P}*_{G_{PQ}}G_{Q}$
has a natural topology generated by the topologies of $G_{P}$ and
$G_{Q}$ (because $G_{PQ}$ is an \emph{open subgroup} of both $G_{P}$
and $G_{Q}$), and endowed with this topology the above map is an
isomorphism of topological groups.
\end{proof}
When $\mathcal{G}_{gen}$ is a tree, we may describe the pair $(\mathcal{G}_{gen},A^{PQ})$
in a different way. Given any compact open subgroup $G\subset A^{PQ}$
and any vertex $v\in\mathcal{G}_{gen}$, the orbit $G.v$ is compact
and discrete (as we gave $\mathcal{G}_{gen}$ the discrete topology)
and is hence finite. Therefore $G$ acts on a finite subtree $T$
of $\mathcal{G}_{gen}$ and hence the action factors through a finite
quotient $H$ of $G$.
\begin{lem}
A finite group $H$ acting on a finite tree $T$ always as a fixed
point (though not necessarily a fixed vertex.)
\end{lem}
\begin{proof}
This is well known and Aaron Bernstein explained the following simple
proof to us.

Let the height $h(v)$ of a vertex $v$ be the maximal distance of
$v$ to any leaf. Any automorphism of $T$ preserves heights. If there
is a unique vertex $v$ of minimal height, we are done, so suppose
there is another vertex $w$ of minimal height. Then $v$ and $w$
must be connected by an edge: if the unique path between them contained
an intermediate vertex $u$, then some thought shows that $h(u)<h(v)$.
As $T$ is a tree, there can be \emph{at most two }vertices of minimal
height. If there are two, then their midpoint is a fixed point for
\emph{any }automorphism of $T$.
\end{proof}
Therefore there must be a point $p\in T$ that is fixed by $H$; here
$T$ is thought of as a topological space. If $p$ were not a vertex
$T$, $H$ would fix the two neighboring vertices of the edge $p$
is on because $H$ respects the coloring of the graph. Therefore $H$
fixes at least one vertex $v$. On the other hand, given any vertex
$v$, the subgroup $G_{v}$ fixing $v$ is a compact open subgroup.
Therefore, the vertices of $\mathcal{G}_{gen}$ are in natural bijective
correspondence with the maximal open compact subgroups $G$ of $A^{PQ}$.
\begin{cor}
\label{corollary:max_compact}If $\mathcal{G}_{gen}$ is a tree, any
maximal compact open subgroup $G$ of $A^{PQ}$ is conjugate to either
$G_{P}$ or $G_{Q}$.
\end{cor}
\begin{proof}
The discussion above shows that every maximal compact open subgroup
$G$ of $A^{PQ}$ is $G_{v}$ for some vertex $v$ of $\mathcal{G}_{gen}$.
The group $G_{v}$ is conjugate in $A^{PQ}$ to $G_{P}$ or $G_{Q}$
by Corollary \ref{Corollary:Colored_edge_symmetric}. Finally, $G_{P}$
is not conjugate to $G_{Q}$ in $A^{PQ}$ because the action of $A^{PQ}$
on $\mathcal{G}_{gen}$ preserves the coloring.
\end{proof}
\begin{rem}
\label{Remark:tree_conjugation_action}If $\mathcal{G}_{gen}$ is
a tree, then the action of $A^{PQ}$ on $\mathcal{G}_{gen}$ is the
conjugation action on the maximal compact subgroups.
\end{rem}
We may similarly describe the adjacency relation in $\mathcal{G}_{gen}$
from the group $A^{PQ}$ when $\mathcal{G}_{gen}$ is a tree. Recall
our standing assumption that the original correspondence $X\leftarrow Z\rightarrow Y$
is minimal (in order for $\mathcal{G}_{gen}$ to not have multiple
edges.) As above, we suppose the correspondence is of type $(d,e)$.
Then a blue vertex $G_{v}$ and a red vertex $G_{w}$ are joined by
an edge if the intersection $G_{v}\cap G_{w}$ (inside of $A^{PQ}$)
has index $d$ inside of $G_{v}$ and index $e$ inside of $G_{w}$.

\section{\label{Section:Symmetric_correspondences}Symmetric Correspondences}
\begin{defn}
\label{Definition:Symmetric_correspondence}A symmetric correspondence
of curves over $k$ is a self-correspondence $X\overset{f}{\leftarrow}Z\overset{g}{\rightarrow}X$
over curves over $k$ such that there is an involution $w\in\text{Aut}(Z)$
with $f\circ w=g$, i.e. $w$ swaps $f$ and $g$. We denote by $w^{*}$
the induced involution on $k(Z)$.

Note that if the correspondence is minimal, $w$ is unique if it exists.
Therefore being symmetric is a property and not a structure of a minimal
correspondence.
\end{defn}
\begin{lem}
\label{Lemma:Lift_involution}Let $X\overset{f}{\leftarrow}Z\overset{g}{\rightarrow}X$
be a symmetric correspondence of curves over $k$ without a core.
Suppose $Z$ is hyperbolic or $k\cong\mathbb{F}$. Any $w\in\text{Aut}(Z)$
that swaps $f$ and $g$ lifts to an automorphism $\tilde{w}$ of
$W_{\infty}$. We denote by $\tilde{w}^{*}$ the associated automorphism
of $E_{\infty}=k(W_{\infty})$.

\begin{proof}
We proceed exactly as in the discussion at the beginning of Section
\ref{Sec:A_Recursive_Tower}: let $W_{f}$ (resp. $W_{g}$) denote
a Galois closure of arrow $f$ (resp. $g$). The automorphism $w$
of $Z$ swaps $f$ and $g$ and hence we can choose an isomorphism
$w_{1}:W_{g}\rightarrow W_{f}$ living over $w$ on $Z$:
\[
\xymatrix{W_{g}\ar[d]\ar[r]^{w_{1}} & W_{f}\ar[d]\\
W\ar[r]^{w} & W
}
\]
Similarly, we can chose an isomorphism $w_{2}:W_{gf}\rightarrow W_{fg}$
living over $w$ on $Z$, again because $w$ swaps the roles of $f$
and $g$. Continuing in this fashion, we get an isomorphism of towers
\[
\tilde{w}:W_{gf\dots}\rightarrow W_{fg\dots}
\]
By Lemma \ref{Lemma:final_towers}, $W_{fg\dots}$ is isomorphic to
$W_{gf}$ as a pro-curve over $W$ and we may think of $\tilde{w}$
as an automorphism of $W_{\infty}$ living over $w\in\text{Aut}(Z)$.
\end{proof}
\end{lem}
\begin{rem}
Another way of phrasing Lemma \ref{Lemma:Lift_involution} is as follows.
If $X\overset{f}{\leftarrow}Z\overset{g}{\rightarrow}X$ is a symmetric
correspondence without a core with $Z$ hyperbolic, then for any choice
of symmetry $w$, the following map is (infinite) Galois.
\[
W_{\infty}\rightarrow Z/<w>
\]
From this perspective, it is clear that the lift $\tilde{w}$ is not
unique.
\end{rem}
\begin{defn}
Let $X\overset{f}{\leftarrow}Z\overset{g}{\rightarrow}X$ be a symmetric
correspondence of curves over $k$ without a core where $Z$ is hyperbolic
or $k\cong\mathbb{F}$. Pick a symmetry $w$ and a lift $\tilde{w}$
to $W_{\infty}$, which exists by Lemma \ref{Lemma:Lift_involution}.
Let $\tilde{w}^{*}$ be the associated automorphism of $E_{\infty}$.
Define $A^{w}\subset\text{Aut}_{k}(E_{\infty})$ be the subgroup generated
by $A^{PQ}$ and $\tilde{w}^{*}$. We give the subgroup $A^{w}\subset A$
the induced topology from $A$.
\end{defn}
\begin{rem}
The notation $A^{w}$ is \emph{a priori} ambiguous as it seems to
depend on a choice of lift $\tilde{w}$. Pick a second lift $\tilde{w}'$
of $w$. Then $\tilde{w}\tilde{w}'$ fixes $Z$ as $w$ was an involution.
In particular, $\tilde{w}^{*}\tilde{w}'^{*}\in\text{Gal}(E_{\infty}/k(Z))\subset A^{PQ}$,
so $A^{w}$ is independent of the choice of lift of $w$.
\end{rem}
\begin{cor}
\label{Corollary:one_transitive}Let $X\overset{f}{\leftarrow}Z\overset{g}{\rightarrow}X$
be a symmetric correspondence of curves over $k$ without a core with
symmetry $w$. Suppose $Z$ is hyperbolic or $k\cong\mathbb{F}$ and
let $\tilde{w}$ be a lift of the symmetry to $W_{\infty}$. Then
$A^{w}$ and hence $A$ acts transitively on the oriented edges of
$\mathcal{G}_{gen}$.
\end{cor}
\begin{proof}
Corollary \ref{Corollary:Colored_edge_symmetric} says that $A^{PQ}$
acts transitively on $\mathcal{G}_{gen}$ subject to the constraint
that the colors of the vertices are preserved. The automorphism $\tilde{w}^{*}\in\text{Aut}(E_{\infty})$
swaps the points $P$ and $Q$. By conjugating we get that $A^{w}$
acts transitively on the edges of $\mathcal{G}_{gen}$, in the usual
sense of remembering the endpoints.
\end{proof}
\begin{cor}
Let $X\overset{f}{\leftarrow}Z\overset{g}{\rightarrow}X$ be a symmetric
correspondence of curves over $k$ without a core with symmetry $w$.
Suppose $Z$ is hyperbolic or $k\cong\mathbb{F}$. Then $A^{PQ}$
is an normal subgroup of index 2 inside of $A^{w}$.
\end{cor}
\begin{proof}
Conjugating by $\tilde{w}^{*}$ swaps $G_{P}$ and $G_{Q}$ and hence
stabilizes $A^{PQ}$. Therefore $A^{PQ}$ is normal inside of $A^{w}$.
Moreover, $(\tilde{w}^{*})^{2}\in A^{PQ}$, so $A^{w}/A^{PQ}$ is
of order 2.
\end{proof}
\begin{defn}
Let $(G,A)$ be a pair where $G$ is a connected graph and $A$ is
a group of automorphisms of $G$. $(G,A)$ is said to be (sharply)
$s$-transitive if $A$ acts (sharply) transitively on all $s$-arcs.
$(G,A)$ is said to be $\infty$-transitive if it is $s$-transitive
for all $s\geq1$.
\end{defn}
In this language, under the hypotheses of Corollary \ref{Corollary:one_transitive}
the pair $(\mathcal{G}_{gen},A)$ is 1-transitive.
\begin{thm}
\label{Theorem:Tutte}(Tutte) Let $G$ be a connected trivalent graph,
$A$ a group of automorphisms of $G$, and $s$ a positive integer.
If $(G,A)$ is $s$-transitive and not $s+1$-transitive, then $(G,A)$
is sharply $s$-transitive.
\end{thm}
\begin{proof}
The proof is exactly the same as in 7.72 in Tutte's book Connectivity
in Graphs \cite{tutte1966connectivity}. Alternatively, see Djokovi\'{c}
and Miller \cite{djokovic1980regular}, Theorem 1, for exactly this
statement.
\end{proof}
\begin{lem}
\label{Lemma:free_33}Let $X\leftarrow Z\rightarrow X$ be a symmetric
type (3,3) correspondence of curves over $k$ without a core with
symmetry $w$. Suppose $Z$ is hyperbolic or $k\cong\mathbb{F}$.
Then the pair $(\mathcal{G}_{gen},A^{w})$ is $\infty$-transitive
and $\mathcal{G}_{gen}$ is a tree.
\end{lem}
\begin{proof}
Suppose $\mathcal{G}_{gen}$ had a cycle. The graph $\mathcal{G}_{gen}$
is infinite by Proposition \ref{Proposition:no_core_infinite_graph}.
Then the pair $(\mathcal{G}_{gen},A^{w})$ is 1-transitive, so there
exists a positive $n$ such that $(\mathcal{G}_{gen},A^{w})$ is $n$-transitive
but not $n+1$-transitive. Therefore, to prove $\mathcal{G}_{gen}$
is a tree it suffices to prove that the pair $(\mathcal{G}_{gen},A^{w})$
is $\infty$-transitive. 

Suppose $(\mathcal{G}_{gen},A^{w})$ was not $\infty$-transitive.
Then there exists a positive integer $n$ such that $(\mathcal{G}_{gen},A^{w})$
is $n$-transitive but not $n+1$-transitive because the graph is
infinite, connected and 1-transitive. Theorem \ref{Theorem:Tutte}
implies that the pair $(\mathcal{G}_{gen},A^{w})$ is then sharply
$n$-transitive, i.e. there exists a unique automorphism in $A^{w}$
sending any $n$-arc to any other $n$-arc. Therefore any automorphism
in $A^{w}$ that fixes any given $n$-arc must be the identity automorphism.
To any $n$-arc $R$ I can associate the field $E_{R}$ which is the
field generated by the images of the points and edges inside of $E_{\infty}$
as in Definition \ref{Definition:subfield_subgraph}. Pick the $n$-arc
$R$ through $P$ so that $E_{R}$ is a finite extension of $P(k(X))$.
Note that $E_{\infty}$ is Galois over $E_{R}$. The group $\text{Gal}(E_{\infty}/E_{R})$
acts faithfully on $\mathcal{G}_{gen}$ and fixes $R$. As $(\mathcal{G}_{gen},A^{w})$
is sharply $n$-transitive, the group $\text{Gal}(E_{\infty}/E_{R})$
acts trivially on $\mathcal{G}_{gen}$. Therefore $E_{R}=E_{\infty}$
is a finite extension $k(Z)$, Galois over both $k(X)$ and $k(Y)$,
which is a contradiction.
\end{proof}
Lemma \ref{Lemma:free_33} poses the following refinement to Question
\ref{Question:no_core_a_tree?} on whether or not $\mathcal{G}_{gen}$
is a tree.
\begin{question}
\label{Question:infty_transitive}Let $X\leftarrow Z\rightarrow X$
be a minimal, symmetric, étale correspondence of curves over $k$
without a core. Is the pair $(\mathcal{G}_{gen},A^{w})$ $\infty$-transitive?
\end{question}
We may use Question \ref{Question:infty_transitive} to pose a refinement
of Question \ref{Question:No_core_Shimura}
\begin{question}
Let $X\leftarrow Z\rightarrow X$ be a minimal, symmetric, étale correspondence
of projective curves over $k$ without a core. Does the pair $(\mathcal{G}_{gen},A^{w})$
\textquotedbl{}look like\textquotedbl{} the action of $SL_{2}$ over
a local field on its building? 
\end{question}

\section{\label{Section:Specialization_of_Graphs}Specialization of Graphs
and Special Orbits}

Given a correspondence $X\overset{f}{\leftarrow}Z\overset{g}{\rightarrow}Y$
over a field $k$, we have defined an undirected 2-colored graph $\mathcal{G}_{gen}^{full}$,
the full generic graph\emph{,} using an algebraically closed overfield
$\Omega$. In this section we define the $\mathcal{G}_{phys}^{full}$,
the\emph{ }full physical graph, which will be an undirected 2-colored
graph, using $\overline{k}$. The goal of this section is to speculate
on the behavior of ``specialization maps'' $s_{\tilde{z}}:\mathcal{G}_{gen}\rightarrow\mathcal{G}_{phys,z}$;
informally, if we think of $\mathcal{G}_{gen}$ as the ``graph of
generic dynamics'', this map specializes to the graph associated
to the dynamics of a physical point $z\in Z(\overline{k})$.
\begin{defn}
Given a correspondence $X\overset{f}{\leftarrow}Z\overset{g}{\rightarrow}Y$
of curves over $k$, the \emph{full physical graph }$\mathcal{G}_{phys}^{full}$
is the following 2-colored graph. The edges are the points $z\in Z(\overline{k})$,
the blue vertices are the points $X(\overline{k})$ and the red vertices
are the points $Y(\overline{k})$. Adjacent to $z:\text{Spec}(\overline{k})\rightarrow Z$
is the blue vertex $f\circ z\in X(\overline{k})$ and the red vertex
$g\circ z\in Y(\overline{k})$. Given a choice of $z\in Z(\overline{k})$,
we denote by the $\mathcal{G}_{phys,z}$ the connected component of
$\mathcal{G}_{phys}^{full}$ that contains $z$.
\end{defn}
Recall the construction of $\mathcal{G}_{gen}$: pick an edge $PQ\in Z(\Omega)$
of $\mathcal{G}_{gen}^{full}$ and define $\mathcal{G}_{gen}$ to
be the connected component of $\mathcal{G}_{gen}^{full}$ that contains
$PQ$, suppressing the implicit $PQ$ in the notation. The field $E_{\infty}\subset\Omega$
is the compositum of all of the points and edges of of $\mathcal{G}_{gen}$,
thought of as subfields of $\Omega$, by Corollary \ref{Corollary:two_defs_e_infinity}.
Therefore, an edge $e$ of $\mathcal{G}_{gen}$ yields an element
of the set $Z(E_{\infty})$. Similarly, a blue vertex $v$ of $\mathcal{G}_{gen}$
yields an element of $X(E_{\infty})$ and a red vertex $w$ yields
an element of $Y(E_{\infty})$.

We spell out exactly what is fixed in the construction of a specialization
map. First of all, assume the curves $X$, $Y$, and $Z$ are proper
over $k$: this is harmless as any correspondence of curves has a
canonical compactification. Pick $z\in Z(\overline{k})$. Then pick
a point $\tilde{z}\in W_{\infty}(\overline{k})$, a geometric point
of the scheme $W_{\infty}$, i.e. a compatible system of geometric
points on the tower defining $W_{\infty}$, lying over $z$. Taking
the image of $\tilde{z}$ gives closed point of the scheme $W_{\infty}$,
and the ring $\mathcal{O}_{W_{\infty},\tilde{z}}$ is a valuation
ring because it is the filtered colimit of valuation rings. Moreover,
the fraction field of $\mathcal{O}_{W_{\infty},\tilde{z}}$ is $E_{\infty}$.
The choice of $\tilde{z}:\text{Spec}(\overline{k})\rightarrow W_{\infty}$
yields a morphism $\pi:\mathcal{O}_{W_{\infty},\tilde{z}}\rightarrow\overline{k}$.
We now construct the specialization map 
\[
s_{\tilde{z}}:\mathcal{G}_{gen}\rightarrow\mathcal{G}_{phys}^{full}
\]
Let $e$ be an edge of $\mathcal{G}_{gen}$. As discussed above, $e$
yields an element of $Z(E_{\infty})$. We want to describe $s_{\tilde{z}}(e_{\infty})$,
the image of $e$, in $\mathcal{G}_{phys,z}$. We have the following
diagram; the dotted arrow exists uniquely because the structure map
$Z\rightarrow\text{Spec}(k)$ is proper. 
\[
\xymatrix{ & \text{Spec}(E_{\infty})\ar[rr]^{e}\ar[d] &  & Z\ar[d]\\
\text{Spec}(\overline{k})\ar[r]^{\tilde{z}} & \text{Spec}(\mathcal{O}_{W_{\infty},\tilde{z}})\ar[rr]\ar@{-->}[urr] &  & \text{Spec}(k)
}
\]
Composing $\tilde{z}$ with the dotted arrow, we get an element $\overline{e}\in Z(\overline{k})$.
We set $s_{\tilde{z}}(e)=\overline{e}$. The exact same construction
works with (red and blue) vertices, and the result is manifestly a
map of graphs. Moreover, as $\mathcal{G}_{gen}$ is connected, so
is the image.

Finally, we show that the edge $PQ\in Z(E_{\infty})$ is sent to $z$.
The inverse image of $\mathcal{O}_{W_{\infty},\tilde{z}}$ under the
map $PQ:k(Z)\hookrightarrow E_{\infty}$is the valuation ring $R$
of $k(Z)$ corresponding to $z$. Therefore, when $e=PQ$, the above
dotted arrow corresponds to the inclusion $R\hookrightarrow\mathcal{O}_{W_{\infty},\tilde{z}}$.
As $\tilde{z}$ lives over $z$, composing this inclusion with $\pi$
yields $s_{\tilde{z}}(PQ)=z$, as desired. Therefore, we have constructed
a map of graphs.
\[
s_{\tilde{z}}:\mathcal{G}_{gen}\rightarrow\mathcal{G}_{phys,z}
\]

\begin{lem}
\label{Lemma:etale_specialization}Let $X\overset{f}{\leftarrow}Z\overset{g}{\rightarrow}Y$
be an étale correspondence of projective hyperbolic curves without
a core over a field $k$. Then all of the specialization maps are
surjective. 
\[
s_{\tilde{z}}:\mathcal{G}_{gen}\rightarrow\mathcal{G}_{phys,z}
\]
\end{lem}
\begin{proof}
Because the correspondence is étale, each blue vertex of $\mathcal{G}_{phys}$
is adjacent to $d=\text{deg}(f)$ edges and each red vertex is adjacent
to $e=\text{deg}(g)$ edges. It is therefore equivalent to show that
no two adjacent edges of $\mathcal{G}_{gen}$ are sent to the same
edge in $\mathcal{G}_{phys,z}$. Let $A$ and $B$ be two edges sharing
the blue vertex $p$. We want to show that $A$ and \textbf{$B$ }are
not sent to the same edge in $\mathcal{G}_{phys,z}$.

Recall that $A$ and $B$ yield elements of $Z(E_{\infty})$ such
that $f\circ A=f\circ B=p\in X(E_{\infty})$. Proposition \ref{Proposition:finite_distance}
implies that, after possibly enlarging $k$, there exists an irreducible
curve $C$ together with maps $\rho:\text{Spec}(E_{\infty})\rightarrow C$,
$\pi:C\rightarrow Z$, and $a,b:C\rightarrow Z$ such that 

\begin{itemize}
\item $\pi\circ\rho=PQ$ considered as elements of $Z(E_{\infty})$
\item $A$ and $B$ factor through $C$ via $a$ and $b$.
\end{itemize}
In the language of Proposition \ref{Proposition:finite_distance},
$C$ is the curve associated to a field $F$ of transcendence degree
1 over $k$, finite over $PQ(k(Z))$, that contains $A(k(Z))$ and
$B(k(Z))$. More explicitly, we have the following factorizations:
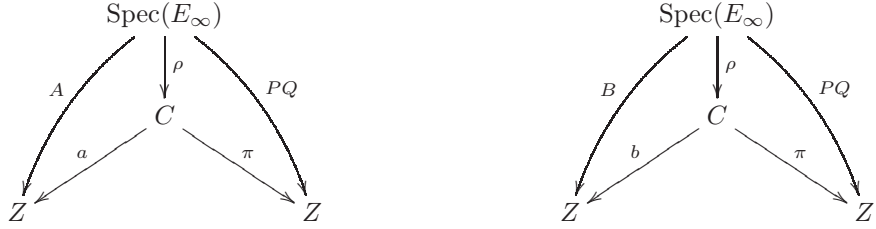
\[
\xymatrix{ & \text{Spec}(E_{\infty})\ar[d]^{\rho}\ar@/^{.75pc}/[ddr]^{PQ}\ar@/_{.75pc}/[ddl]_{A} &  &  &  &  & \text{Spec}(E_{\infty})\ar[d]^{\rho}\ar@/^{.75pc}/[ddr]^{PQ}\ar@/_{.75pc}/[ddl]_{B}\\
 & C\ar[dr]^{\pi}\ar[dl]_{a} &  &  &  &  & C\ar[dr]^{\pi}\ar[dl]_{b}\\
Z &  & Z &  &  & Z &  & Z
}
\]
Moreover, the maps $\pi$, $a$, and $b$ are all finite étale. Let
us follow the specialization construction. Again, the dotted arrow
exists because $C\rightarrow\text{Spec}(k)$ is proper. 
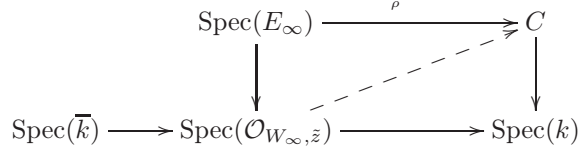
\[
\xymatrix{ & \text{Spec}(E_{\infty})\ar[rr]^{^{\rho}}\ar[d] &  & C\ar[d]\\
\text{Spec}(\overline{k})\ar[r] & \text{Spec}(\mathcal{O}_{W_{\infty},\tilde{z}})\ar[rr]\ar@{-->}[urr] &  & \text{Spec}(k)
}
\]
This diagram gives us a point $x\in C(\overline{k})$ by composition
with the dotted arrow. If $A$ and $B$ are identified under the specialization
map, $a(x)=b(x)\in Z(\overline{k})$. Now, $f\circ a=f\circ b$ because
$A$ and $B$ shared the vertex $p$, so we have the following diagram.
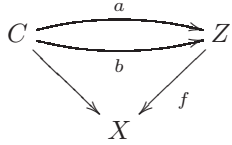
\[
\xymatrix{C\ar@/_{.5pc}/[rr]_{b}\ar@/^{.5pc}/[rr]^{a}\ar[dr] &  & Z\ar[dl]^{f}\\
 & X
}
\]
But $C$ is irreducible and the maps $a$, $b$, and $f$ are finite
étale, so the assumption that $a(x)=b(x)$ implies that $a=b$ and
hence $A=B$, as desired.

\end{proof}
\begin{rem}
\label{Remark:W_infty_and_G_phys}The graph $\mathcal{G}_{phys}$
helps describe the tower $W_{\infty}$. In this remark, we suppose
all morphisms are unramified at all points specified. For instance,
let $\xi_{Y}\in W_{Y}(\overline{k})$ map to $z\in Z(\overline{k})$
which maps to $y\in Y(\overline{k})$. Then, as in Remark \ref{Remark:Many maps from W_Y to Z},
there are naturally $\text{Aut}(W_{Y}/Z)$ many maps from $W_{Y}$
to $Z$ and we can look at the images of $\xi_{Y}$ under these maps.
In this way, $\xi_{Y}$ yields the graph of all edges coming out of
$y$ in $\mathcal{G}_{phys,z}$. More generally, a point $\xi_{YX\dots Y}\in W_{YX\dots Y}(\overline{k})$
which maps to $y\in Y(\overline{k})$ under the natural map yields
the subgraph of $\mathcal{G}_{phys,z}$ with center $y$ and radius
$n$, where $n$ is the number of letters in the string ``$YX\dots Y$''.

We will use this Remark in Proposition \ref{Proposition:if_clump_then_bounded_field_of_constants}
to show that if a \emph{étale clump }exists, then Question \ref{Question:Field_of_Constants}
has an affirmative answer.
\end{rem}
Consider the Hecke correspondence of open modular curves over $\mathbb{F}_{p}$
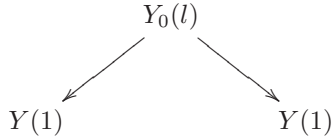
\[
\xymatrix{ & Y_{0}(l)\ar[dl]\ar[dr]\\
Y(1) &  & Y(1)
}
\]
The graph $\mathcal{G}_{gen}$ is a tree. For $z\in Y_{0}(l)(\mathbb{F})$
an ordinary point, $\mathcal{G}_{phys,z}$ has at most one cycle\emph{.
}This follows from the work in David Kohel's thesis \cite{kohel1996endomorphism},
summarized by Andrew Sutherland \cite{sutherland2013isogeny}. They
call this structure an \emph{Isogeny Volcano. }The cycle comes from
the following fact: given an imaginary quadratic field $K/\mathbb{Q}$,
there exists an elliptic curve $E/\mathbb{F}$ with multiplication
by the maximal order $\mathcal{O}_{K}$. On the other hand, there
are only finitely many supersingular points, and in fact Theorem \ref{Theorem:one_clump}
implies that if $\mathcal{G}_{phys,z}$ contains one supersingular
point it contains all of them.
\begin{defn}
Given an étale correspondence of projective hyperbolic curves 
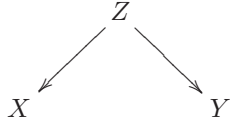
\[
\xymatrix{ & Z\ar[dl]\ar[dr]\\
X &  & Y
}
\]
over $k$ without a core, we say a point $z\in Z(\overline{k})$ is
\emph{special} if there exists (equivalently for all) $\tilde{z}\in W_{\infty}(\overline{k})$
over $z$ such that the map $s_{\tilde{z}}:\mathcal{G}_{gen}\rightarrow\mathcal{G}_{phys,z}$
is not an isomorphism. We say $z\in Z(\overline{k})$ is \emph{generic}
if the it is not special.
\end{defn}
\begin{question}
\label{Question:Special_orbit}Let $X\leftarrow Z\rightarrow Y$ be
an étale correspondence of projective curves over $\mathbb{F}_{q}$
without a core . 

\begin{enumerate}
\item Is there always $z\in Z(\mathbb{F})$ that is generic?
\item Is there always a special point with unbounded orbit?
\item Suppose $\mathcal{G}_{gen}$ is free. For every point $z\in Z(\mathbb{F})$,
is $\pi_{1}(\mathcal{G}_{phys,z})$ finitely generated? If $\mathcal{G}_{phys,z}$
is infinite, does $\mathcal{G}_{phys,z}$ have one cycle?
\item What is $\lim_{n\rightarrow\infty}\frac{|z\in Z(\mathbb{F}_{q^{n}})\text{ with }z\text{ generic}|}{|Z(\mathbb{F}_{q^{n}})|}$?
\end{enumerate}
\end{question}

\section{\label{Section:Invariant_Line_Bundles}Invariant Line Bundles and
Invariant Sections}

In this section we will need somewhat refined information about abelian
varieties and finite group schemes over a field $k$. Our main reference
is van der Geer and Moonen \cite{van2007abelian}. 
\begin{defn}
Let $X\overset{f}{\leftarrow}Z\overset{g}{\rightarrow}Y$ be a correspondence
of curves over $k$. An \emph{invariant line bundle} $\mathscr{L}$
on the correspondence is a triple $(\mathscr{L}_{X},\mathscr{L}_{Y},\phi)$
where $\mathscr{L}_{X}$ is a line bundle on $X$, $\mathscr{L}_{Y}$
is a line bundle on $Y$, and $\phi:f^{*}\mathscr{L}_{X}\rightarrow g^{*}\mathscr{L}_{Y}$
is an isomorphism of line bundles on $Z$. The \emph{degree} of an
invariant line bundle $\mathscr{L}$ on a correspondence of projective
curves is $\text{deg}(f^{*}\mathscr{L}_{X})=\text{deg}(g^{*}\mathscr{L}_{Y})$
on $Z$. An isomorphism of invariant line bundles $i:\mathscr{L}\rightarrow\mathscr{L}'$
is a pair of isomorphisms $i_{X}:\mathscr{L}_{X}\rightarrow\mathscr{L}'_{X}$
and $i_{Y}:\mathscr{L}_{Y}\rightarrow\mathscr{L}'_{Y}$ that intertwine
$\phi$ and $\phi'$ when pulled back to $Z$. 

The cohomology of an invariant line bundle $\mathscr{L}$ is defined
as follows. 
\[
H^{i}(\mathscr{L}):=\{(\xi_{X},\xi_{Y})\in H^{i}(X,\mathscr{L}_{X})\oplus H^{i}(Y,\mathscr{L}_{Y})|f^{*}(\xi_{X})=\phi^{*}g^{*}(\xi_{Y})\in H^{i}(Z,f^{*}\mathscr{L}_{X})\}
\]
The group $H^{i}(\mathscr{L})$ is naturally a finite dimensional
$k$ vector space, and we let $h^{i}(\mathscr{L})=\text{dim}_{k}H^{i}(\mathscr{L})$.
We call elements of $H^{0}(\mathscr{L})$ \emph{invariant sections}.
When it is especially clear from context, we omit the prefix ``invariant''.
\end{defn}
In general, $\mathcal{O}=(\mathcal{O}_{X},\mathcal{O}_{Y},1)$ is
an invariant line bundle. Note that if the correspondence is étale,
there is a natural invariant line bundle: $\Omega=(\Omega_{X}^{1},\Omega_{Y}^{1},\phi)$;
here $\phi$ is the composition of the canonical isomorphism $f^{*}\Omega_{X}^{1}\rightarrow\Omega_{Z}^{1}$
and the inverse of the canonical isomorphism $g^{*}\Omega_{Y}^{1}\rightarrow\Omega_{Z}^{1}$.
We call elements of $H^{0}(\Omega)$ \emph{invariant differential
1-forms}. Let $\mathscr{T}$ denote the dual of $\Omega$. Then the
first-order deformation space of an étale correspondence of projective
curves is $H^{1}(\mathscr{T})$.
\begin{prop}
\label{Proposition:at_most_one_section}Let $X\overset{f}{\leftarrow}Z\overset{g}{\rightarrow}Y$
be a correspondence of curves over $k$ without a core. Let $\mathscr{L}$
be an invariant line bundle on the correspondence. Then $h^{0}(\mathscr{L})\leq1$.
\end{prop}
\begin{proof}
If there were two linearly independent sections $s=(s_{X},s_{Y})$
and $t=(t_{X},t_{Y})$, then by taking their ratio we get a map to
$\mathbb{P}^{1}$.
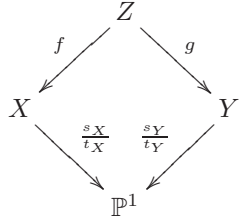
\[
\xymatrix{ & Z\ar[dl]_{f}\ar[dr]^{g}\\
X\ar[dr]^{\frac{s_{X}}{t_{X}}} &  & Y\ar[dl]_{\frac{s_{Y}}{t_{Y}}}\\
 & \mathbb{P}^{1}
}
\]
Hence there is a core.
\end{proof}
\begin{question}
Let $X\leftarrow Z\rightarrow Y$ be an étale correspondence of projective
curves over $k$ without a core, where $\text{char}(k)=p$. Suppose
$X$ has genus $g$. What is the maximal dimension of $h^{1}(\mathscr{T})$
in terms of $g$?
\end{question}
\begin{rem}
As noted in Remark \ref{Remark:no_defs}, in characteristic 0 étale
correspondences without a core do not deform. However, Example \ref{Example:Central_Leaf}
shows that in characteristic $p$ they may deform .
\end{rem}
We will see that, in characteristic 0, there are no invariant sections
of any non-trivial invariant line bundle on an étale correspondence
without a core. However, they can exist in characteristic $p$. To
better understand this, we briefly review the \emph{cyclic cover trick}
for smooth curves. Let $T$ be a smooth curve over $k$, let $\mathscr{L}_{T}$
be a line bundle on $T$. Suppose $s\in H^{0}(T,\mathscr{L}_{T}^{d})$
for some $d\in\mathbb{N}$ with $(d,\text{char}k)=1$, such that $s$
is not the power of a section of a smaller power of $\mathscr{L}_{T}$.
Let $\mathscr{A}_{s}$ denote the following sheaf of algebras 
\[
\mathscr{A}_{s}=\mathcal{O}_{T}\oplus\mathscr{L}_{T}^{-1}\oplus\dots\mathscr{L}_{T}^{-(d-1)}
\]
with multiplication given by the naive multiplication when possible
and contraction with $s$ when necessary. The condition that $s$
is not the power of a section implies that $\mathscr{A}_{s}$ is an
irreducible sheaf of algebras. We let $T(s^{\frac{1}{d}})\rightarrow T$,
the \emph{$d$\textsuperscript{\emph{th}}-cyclic cover of $T$ by
$s$}, be the normalization of $\text{Spec}_{T}\mathscr{A}_{s}$ equipped
with the natural map to $T$. Then $T(s^{\frac{1}{d}})$ is a smooth
curve over $k$. Then the pullback of $\mathscr{L}_{T}$ to $T(s^{\frac{1}{d}})$
has a (non-canonical) section, $s^{\frac{1}{d}}$, whose $d^{\text{th}}$
power is $s$.

We remark that, by construction, the \emph{$d$\textsuperscript{th}}
cyclic cover of $(T,s)$ is functorial. In particular, let $\mathscr{L}$
be an invariant line bundle on $X\leftarrow Z\rightarrow Y$ with
$s\in H^{0}(\mathscr{L}^{d})$ an invariant section, and suppose that
$s$ is not the power of any invariant section of a smaller power
of $\mathscr{L}$. Then we may perform the cyclic cover trick to $(X\leftarrow Z\rightarrow Y,s)$
to obtain
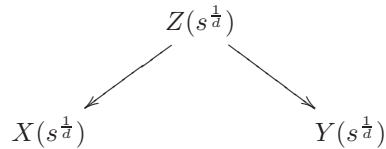
\[
\xymatrix{ & Z(s^{\frac{1}{d}})\ar[dl]\ar[dr]\\
X(s^{\frac{1}{d}}) &  & Y(s^{\frac{1}{d}})
}
\]
The pullback of $\mathscr{L}$ to $X(s^{\frac{1}{d}})\rightarrow Z(s^{\frac{1}{d}})\rightarrow Y(s^{\frac{1}{d}})$
has a (non-canonical) invariant section, which we denote by $s^{\frac{1}{d}}$.
\begin{example}
\label{Example:igusa_form}Consider a Hecke correspondence of (open)
modular curves over $\mathbb{F}$. Then $\Omega^{p-1}$ is an invariant
line bundle and has an invariant section $H$: the Hasse invariant.
Recall that the divisor of $H$ is the supersingular locus. The Hasse
invariant similarly exists on a Hecke correspondence of moduli spaces
of fake elliptic curves. Therefore, there are examples of invariant
sections of non-trivial invariant line bundles on étale correspondences
of projective curves over $\mathbb{F}$ without a core.

In these cases, the \textquotedbl{}Igusa level structure\textquotedbl{}
construction of Remark \ref{Remark:Igusa} is precisely the \emph{$(p-1)$\textsuperscript{st}}cyclic
cover construction associated to the invariant section $H$ of $\Omega^{p-1}$.
In particular, the induced correspondences of Igusa curves have an
invariant differential form: $H^{\frac{1}{p-1}}$. See Ullmer \cite{ulmer1990universal}
for a brief introduction to the Hasse invariant and the Igusa construction
and Chapter 1 of Katz \cite{katz1973p} for a more thorough explication
of modular forms.
\end{example}
Given a correspondence of projective curves, $X\leftarrow Z\rightarrow Y$,
there are induced maps $f^{*}:Pic(X)\rightarrow Pic(Z)$ and $g^{*}:Pic(Y)\rightarrow Pic(Z)$
between the Picard schemes; both of these maps have finite (though
not-necessarily reduced) kernel. Restricting, there are induced maps
$f^{*}:Pic^{0}(X)\rightarrow Pic^{0}(Z)$ and $g^{*}:Pic^{0}(Y)\rightarrow Pic^{0}(Z)$.
We denote by $f^{*}Pic^{0}(X)\cap g^{*}Pic^{0}(Y)$ the scheme-theoretic
intersection of the image of these two maps in $Pic^{0}(Z)$; note
that this group scheme need not be reduced in positive characteristic.
\begin{defn}
Let $X\overset{f}{\leftarrow}Z\overset{g}{\rightarrow}Y$ be a correspondence
of projective curves over $k$. The \emph{Picard scheme of $X\leftarrow Z\rightarrow Y$}
is the closed subgroup scheme of $Pix(X)\times Pic(Y)$ given by
\[
Pic(X\leftarrow Z\rightarrow Y):=\text{ker}(Pic(X)\times Pic(Y)\overset{f^{*}-g^{*}}{\rightarrow}Pic(Z))
\]
 Similarly, $Pic^{0}(X\leftarrow Z\rightarrow Y):=\text{ker}(Pic^{0}(X)\times Pic^{0}(Y)\overset{f^{*}-g^{*}}{\rightarrow}Pic^{0}(Z))$.
\end{defn}
\begin{rem}
The scheme $Pic^{0}(X\leftarrow Z\rightarrow Y)$ need not be reduced
in positive characteristic. As usual, if $Z$ has a $k$-rational
point, then $Pic(X\leftarrow Z\rightarrow Y)(k)$ is isomorphic the
group of isomorphism classes of invariant line bundles on $X\leftarrow Z\rightarrow Y$.
Finally, $Pic(X\leftarrow Z\rightarrow Y)/Pic^{0}(X\leftarrow Z\rightarrow Y)\hookrightarrow\mathbb{Z}$
via the degree map on $Z$.
\end{rem}
We note that $Pic(X\leftarrow Z\rightarrow Y)\rightarrow Pic(X)$
and $Pic(X\leftarrow Z\rightarrow Y)\rightarrow Pic(Y)$ both have
finite kernels. Moreover, 
\[
Pic^{0}(X\leftarrow Z\rightarrow Y)\rightarrow f^{*}Pic^{0}(X)\cap g^{*}Pic^{0}(Y)\subset Pic^{0}(Z)
\]
has finite kernel. The following theorem will be very useful for us.
\begin{thm}
\label{Theorem:closed_reduced_AV}Let $A$ be an abelian variety over
a field $k$ and let $G\hookrightarrow A$ be a closed subgroup scheme.
Then the connected reduced group subscheme $G_{red}^{0}\hookrightarrow A$
is an abelian subvariety.
\end{thm}
\begin{proof}
This is Proposition 5.31 in \cite{van2007abelian}.
\end{proof}
\begin{lem}
\label{Lemma:No_abelian_subvariety}Let $X\overset{f}{\leftarrow}Z\overset{g}{\rightarrow}Y$
be a correspondence of projective curves over $k$ without a core.
Then $Pic(X\leftarrow Z\rightarrow Y)$ has no positive-dimensional
abelian subvarieties. In particular, $Pic^{0}(X\leftarrow Z\rightarrow Y)$
is a finite group scheme over $k$.
\end{lem}
\begin{proof}
We may suppose all of the curves have genus at least 1. It is equivalent
to prove that there is no abelian variety $A$ with finite maps fitting
into the following diagram
\[
\xymatrix{ & Pic^{0}(Z)\\
Pic^{0}(X)\ar[ur]^{f^{*}} &  & Pic^{0}(Y)\ar[ul]_{g^{*}}\\
 & A\ar[ul]\ar[ur]
}
\]

By dualizing, this is equivalent to showing that there is no abelian
variety $B$ with non-constant surjective maps fitting into the following
diagram
\[
\xymatrix{ & JZ\ar[dl]_{f_{*}}\ar[dr]^{g_{*}}\\
JX\ar[dr] &  & JY\ar[dl]\\
 & B
}
\]
(While $Pic^{0}(Z)$ is canonically principally polarized, we write
the dual as $JZ$ to remember the Albanese functoriality.) Suppose
such a $B$ fitting into the diagram existed. We will prove that the
correspondence has a core. Choose a point $z\in Z(k)$ (extend $k$
if necessary) and let $x=f(z)$, $y=g(z)$. Then we have Abel-Jacobi
maps which yield a morphism of \emph{correspondences}:
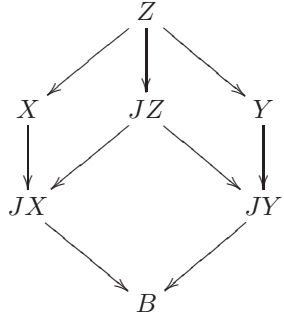
\[
\xymatrix{ & Z\ar[dl]\ar[dr]\ar[d]\\
X\ar[d] & JZ\ar[dl]\ar[dr] & Y\ar[d]\\
JX\ar[dr] &  & JY\ar[dl]\\
 & B
}
\]
i.e. the above diagram commutes. Moreover, under the Abel-Jacobi map,
$Z$ spans $JZ$ as a group and likewise with $X$ and $Y$. Therefore
the induced maps $X\rightarrow B$ and $Y\rightarrow B$ are non-constant.
In particular, their image in $B$ is a curve; therefore $X\leftarrow Z\rightarrow Y$
has a core.

We now prove that $Pic^{0}(X\leftarrow Z\rightarrow Y)$ is finite.
If $Pic^{0}(X\leftarrow Z\rightarrow Y)$ were not finite, then it
would be a positive-dimensional group subscheme of $Pic^{0}(X)\times Pic^{0}(Y)$.
Then $A=Pic^{0}(X\leftarrow Z\rightarrow Y)_{red}^{0}$ is a closed,
reduced, connected subgroup scheme of an abelian variety over $k$
and is hence an abelian variety by Theorem \ref{Theorem:closed_reduced_AV}.
\end{proof}
\begin{cor}
\label{Corollary:high_power_is_pluricanonical}Let $X\leftarrow Z\rightarrow Y$
be an étale correspondence of projective hyperbolic curves over $k$
without a core. Let $\mathscr{L}$ be an invariant line bundle of
positive degree. Then there exists $j,k\in\mathbb{N}$ with $\mathscr{L}^{j}\cong\Omega^{k}$.
\end{cor}
\begin{proof}
The set of degrees of invariant line bundles is a subgroup of $\mathbb{Z}$,
so $\Omega^{-n}\otimes\mathscr{L}^{m}$ has degree 0 for some $m,n\in\mathbb{N}$.
As our correspondence doesn't have a core, $\Omega^{-n}\otimes\mathscr{L}^{m}$
is torsion by Lemma \ref{Lemma:No_abelian_subvariety}. Therefore
there exists $j,k\in\mathbb{N}$ such that $\mathscr{L}^{j}\cong\Omega^{k}$. 
\end{proof}
Corollary \ref{Corollary:high_power_is_pluricanonical} shows that,
for étale correspondences of projective curves without a core, $\Omega$
plays a special role. We will now see several striking consequences
of Lemma \ref{Lemma:No_abelian_subvariety} and Corollary \ref{Corollary:high_power_is_pluricanonical}
in characteristic 0. 
\begin{cor}
\label{Corollary:no_intersection_hodge}Let $X\overset{f}{\leftarrow}Z\overset{g}{\rightarrow}Y$
be a correspondence of projective curves over $k$ without a core.
Suppose $\text{char}(k)=0$. Then $f^{*}H^{1}(X,\mathcal{O}_{X})\cap g^{*}H^{1}(Y,\mathcal{O}_{Y})=0$
inside of $H^{1}(Z,\mathcal{O}_{Z})$ and $f^{*}H^{0}(X,\Omega_{X}^{1})\cap g^{*}H^{1}(Y,\Omega_{Y}^{1})=0$
inside of $H^{0}(Z,\Omega_{Z}^{1})$.
\end{cor}
\begin{proof}
The vector space $H^{1}(X,\mathcal{O}_{X})$ is the tangent space
at the identity of $Pic^{0}(X)$. Moreover, the vector space $f^{*}H^{1}(X,\mathcal{O}_{X})\cap g^{*}H^{1}(Y,\mathcal{O}_{Y})$
is the tangent space at the identity of $f^{*}Pic^{0}(X)\cap g^{*}Pic^{0}(Y)$,
a closed subgroup of $Pic^{0}(Z)$. As the characteristic is 0, $f^{*}Pic^{0}(X)\cap g^{*}Pic^{0}(Y)$
is reduced and hence the connected component of the identity of $f^{*}Pic^{0}(X)\cap g^{*}Pic^{0}(Y)$
is an abelian variety. Lemma \ref{Lemma:No_abelian_subvariety} implies
that this abelian variety has dimension 0 and hence $f^{*}H^{1}(X,\mathcal{O}_{X})\cap g^{*}H^{1}(Y,\mathcal{O}_{Y})=0$
.

By the Lefschetz principle, we may suppose $k\cong\mathbb{C}$. If
$C$ is a smooth projective complex curve, $H_{sing}^{1}(C(\mathbb{C}),\mathbb{C})\cong H^{0}(C,\Omega_{C}^{1})\oplus H^{1}(C,\mathcal{O}_{C})$
and $\overline{H^{1}(C,\mathcal{O}_{C})}=H^{0}(C,\Omega_{C}^{1})$
by Hodge symmetry. Therefore 
\[
f^{*}H^{0}(X,\Omega_{X}^{1})\cap g^{*}H^{0}(Y,\Omega_{Y}^{1})=\overline{f^{*}H^{1}(X,\mathcal{O}_{X})\cap g^{*}H^{1}(Y,\mathcal{O}_{Y})}
\]
inside of $H_{sing}^{1}(Z(\mathbb{C}),\mathbb{C})$. The fact that
$\text{dim}f^{*}H^{1}(X,\mathcal{O}_{X})\cap g^{*}H^{1}(Y,\mathcal{O}_{Y})=0$
implies the result.
\end{proof}
\begin{cor}
\label{Corollary:no_intersection_sing}Let $X\overset{f}{\leftarrow}Z\overset{g}{\rightarrow}Y$
be a correspondence of projective curves over $\mathbb{C}$ without
a core. Then $f^{*}H_{sing}^{1}(X,\mathbb{Z})\cap g^{*}H_{sing}^{1}(Y,\mathbb{Z})=0$
inside of $H_{sing}^{1}(Z,\mathbb{Z})$.
\end{cor}
\begin{proof}
This is immediate from Corollary \ref{Corollary:no_intersection_hodge}
and the fact that pulling back $H_{sing}^{1}$ under $f$ and $g$
induces a morphism of integral Hodge structures. 
\end{proof}
\begin{cor}
\label{Corollary:no_invariant_sections_char_0}Let $X\leftarrow Z\rightarrow Y$
be an étale correspondence of projective curves over $k$ without
a core. Suppose $\text{char}(k)=0$. Let $\mathscr{L}$ be a non-trivial
invariant line bundle. Then $h^{0}(\mathscr{L})=0$. 
\end{cor}
\begin{proof}
We may suppose $\text{deg}\mathscr{L}>0$. Then there exists $j,k\in\mathbb{N}$
such that $\mathscr{L}^{j}\cong\Omega^{k}$ by Corollary \ref{Corollary:high_power_is_pluricanonical}.
It therefore suffices to prove that no positive power of $\Omega$
has a section.

Suppose $s\in H^{0}(\Omega^{k})$ that is not the power of any smaller-degree
invariant pluricanonical form on $X\leftarrow Z\rightarrow Y$. Then
we may apply the cyclic-cover trick to obtain an étale correspondence
\[
\xymatrix{ & Z(s^{\frac{1}{k}})\ar[dl]\ar[dr]\\
X(s^{\frac{1}{k}}) &  & Y(s^{\frac{1}{k}})
}
\]
with an invariant differential form. This contradicts Corollary \ref{Corollary:no_intersection_hodge}.
\end{proof}
We know, via the example of the Hasse invariant, that Corollary \ref{Corollary:no_invariant_sections_char_0}
is false in characteristic $p$. By examining the argument of Proposition
\ref{Corollary:no_intersection_hodge}, we see that the characteristic
0 hypothesis is used twice. First, we used that fact that all group
schemes are reduced to argue that $h^{1}(\mathcal{O})=0$. Second,
we used the Lefschetz principle and Hodge theory, namely $H^{0}(X,\Omega_{X}^{1})=\overline{H^{1}(X,\mathcal{O}_{X})}$,
to relate $h^{1}(\mathcal{O})$ to $h^{0}(\Omega)$.

We further investigate the failure of Proposition \ref{Corollary:no_invariant_sections_char_0}
in characteristic $p$. To do this, we briefly recall a few facts
about (commutative) finite group schemes. Let $k$ be a field of characteristic
$p$ and let $G$ be a finite group scheme over $k$. We denote by
$G^{0}$ the connected component of the identity. There is the connected-étale
sequence
\[
1\rightarrow G^{0}\rightarrow G\rightarrow G^{\acute{e}t}\rightarrow1
\]
which splits if $k$ is perfect. The space of invariant differentials
on $G$, a $k$-vector space denoted by $\omega_{G/k}$, may be identified
with the cotangent space at the origin of $G$ (see 3.14, 3.15 of
\cite{van2007abelian}). We denote by $G[p]$ the $p$-torsion of
$G$. Then the embedding $G[p]\hookrightarrow G$ induces an isomorphism
on the level of (co)tangent spaces at the identity (e.g. see the proof
of 4.47 of \emph{loc. cit.}) 
\begin{rem}
The nomenclature ``invariant differential'' is slightly overloaded;
we use this phrase to refer to a (left-) invariant differential form
on a group scheme. When we use \textquotedbl{}invariant differential
form\textquotedbl{}, we mean a section of $H^{0}(\Omega)$ on an étale
correspondence. We trust that this is not too confusing. 
\end{rem}
We record the following fact, surely well-known, for lack of a reference.
\begin{lem}
\label{Lemma:separable_diff_forms}Let $f:A\rightarrow B$ be a surjective
morphism of abelian varieties over a field $k$. Then $f$ is separable
if and only if the pullback map $f^{*}:H^{0}(B,\Omega_{B}^{1})\rightarrow H^{0}(A,\Omega_{A}^{1})$
is injective. 
\end{lem}
\begin{proof}
Let $K=\text{ker}(f)$ be the kernel of $f$. Then we have the following
inclusion of group schemes
\[
(K^{0})_{red}\subset K^{0}\subset K
\]
Now, $(K^{0})_{red}$ is a closed, reduced, connected subgroup scheme
of an abelian variety over $k$; hence it is an abelian variety by
Theorem \ref{Theorem:closed_reduced_AV}. Therefore $A/(K^{0})_{red}$
exists as an abelian variety (section 9.5 of Polishchuk \cite{polishchuk2003abelian}
or Example 4.40 in \cite{van2007abelian}.) Similarly, $A/K^{0}$,
a quotient of $A/(K^{0})_{red}$ by the finite group scheme $K^{0}/(K^{0})_{red}$
exists as an abelian variety. We have the following commutative diagram
\[
\xymatrix{K\ar[r] & A\ar[r]\ar[ddr] & B\\
 &  & A/K^{0}\ar[u]\\
 &  & A/(K^{0})_{red}\ar[u]
}
\]
where the right vertical arrows are isogenies. In particular $A/K^{0}\rightarrow B$
is a separable isogeny and $A/(K^{0})_{red}\rightarrow A/K^{0}$ is
a purely inseparable isogeny. By looking at tangent spaces, we see
$K^{0}$ is non-reduced if and only if the pullback map $H^{0}(B,\Omega_{B}^{1})\rightarrow H^{0}(A/(K^{0})_{red},\Omega_{A/(K^{0})_{red}}^{1})$
is not injective. On the other hand the short exact sequence of abelian
varieties over $k$
\[
0\rightarrow(K^{0})_{red}\rightarrow A\rightarrow A/(K^{0})_{red}\rightarrow0
\]
shows that the pullback map $H^{0}(A/(K^{0})_{red},\Omega_{A/(K^{0})_{red}}^{1})\rightarrow H^{0}(A,\Omega_{A}^{1})$
is injective. Therefore $f^{*}:H^{0}(B,\Omega_{B}^{1})\rightarrow H^{0}(A,\Omega_{A}^{1})$
is injective if and only if $K^{0}$ is reduced, i.e. if and only
if $f$ is a separable morphism.
\end{proof}
\begin{cor}
\label{Corollary:map_of_curves_jacobians_separable}Let $f:C\rightarrow D$
be a generically separable, finite morphism of projective curves over
$k$. Then $f_{*}:JC\rightarrow JD$ is separable.
\end{cor}
\begin{proof}
Choose an element $c$ of $C(k)$ (after possibly extending $k$)
and let $d=f(c)$. Then we have the following commutative diagram
\[
\xymatrix{C\ar[r]\ar[d] & JC\ar[d]\\
D\ar[r] & JD
}
\]
where the horizontal arrows are the Abel-Jacobi maps associated to
$c$ and $d$ respectively. Pulling back along these Abel-Jacobi maps
yields isomorphisms $H^{0}(JC,\Omega_{JC}^{1})\rightarrow H^{0}(C,\Omega_{C}^{1})$
and $H^{0}(JD,\Omega_{JD}^{1})\rightarrow H^{0}(D,\Omega_{D}^{1})$,
\emph{compatible with pulling back along $f$ and $f_{*}$}. As $f$
was assumed to be generically separable, we obtain that
\[
(f_{*})^{*}:H^{0}(JD,\Omega_{JD}^{1})\rightarrow H^{0}(JC,\Omega_{JC}^{1})
\]
is injective. Now apply Lemma \ref{Lemma:separable_diff_forms}.
\end{proof}
Let $k$ be a field of characteristic $p$ and let $X\leftarrow Z\rightarrow Y$
be a correspondence of projective curves over $k$ without a core.
Suppose $Pic^{0}(X\leftarrow Z\rightarrow Y)$ is non-trivial. We
have the following diagram
\[
\xymatrix{ & Pic^{0}(Z)\\
Pic^{0}(X)\ar[ur] &  & Pic^{0}(Y)\ar[ul]\\
 & Pic^{0}(X\leftarrow Z\rightarrow Y)\ar[ul]\ar[ur]
}
\]
Let $G=Pic^{0}(X\leftarrow Z\rightarrow Y)[p]$. Take $p$-torsion
and apply Cartier duality to obtain the diagram:
\begin{equation}
\xymatrix{ & JZ[p]\ar[dl]^{f_{*}}\ar[dr]_{g_{*}}\\
JX[p]\ar[dr] &  & JY[p]\ar[dl]\\
 & \check{G}
}
\label{Diagram:commutative_finite_gp_schemes}
\end{equation}

Now, if $A$ is any abelian variety over $k$, the natural inclusion
$A[p]\hookrightarrow A$ induces an isomorphism on the level of invariant
differentials: $H^{0}(A,\Omega_{A}^{1})\cong\omega_{A[p]/k}$. On
the other hand, pulling back differential 1-forms under $f_{*}:JZ\rightarrow JX$
and $g_{*}:JZ\rightarrow JY$ is an injective operation by Corollary
\ref{Corollary:map_of_curves_jacobians_separable}. Therefore pulling
back invariant differentials is injective:
\[
\omega_{JX[p]/k}\hookrightarrow\omega_{JZ[p]/k}
\]
\[
\omega_{JY[p]/k}\hookrightarrow\omega_{JZ[p]/k}
\]
Pick $z\in Z(k)$ and set $x=f(z)$, $y=g(z)$. Then using the compatible
Abel-Jacobi maps, we obtain that the following two vector spaces are
isomorphic.
\[
\{(\eta_{X},\eta_{Y})\in H^{0}(X,\Omega_{X}^{1})\oplus H^{0}(Y,\Omega_{Y}^{1})|f^{*}\eta_{X}=g^{*}\eta_{Y}\}\cong
\]
\[
\{(s,t)\in\omega_{JX[p]/k}\oplus\omega_{JY[p]/k}|f^{*}s=g^{*}t\}
\]

\begin{cor}
Let $X\leftarrow Z\rightarrow Y$ be an étale correspondence of projective
curves over $k$ without a core. Then
\begin{itemize}
\item $h^{0}(\Omega)=\dim_{k}\{(s,t)\in\omega_{JX[p]/k}\oplus\omega_{JY[p]/k}|f^{*}s=g^{*}t\}$
\item If the map $T_{e}JX[p]\rightarrow T_{e}\check{G}$ is non-zero, then
$h^{0}(\Omega)=1$.
\item The dimension of the image of $T_{e}JX[p]\rightarrow T_{e}\check{G}$
is no greater than 1.
\end{itemize}
\end{cor}
\begin{proof}
The first part follows from the above discussion. If the map $T_{e}JX[p]\rightarrow T_{e}\check{G}$
is non-zero, then the pullback map $\omega_{\check{G}/k}\rightarrow\omega_{JX[p]/k}\hookrightarrow\omega_{JZ[p]/k}$
has non-zero image. By the commutativity of Diagram \ref{Diagram:commutative_finite_gp_schemes}
there exists a pair $(s,t)\in\omega_{JX[p]/k}\oplus\omega_{JY[p]/k}$
such that the pullbacks to $\omega_{JZ[p]/k}$ agree . Hence there
exists an invariant differential form on $X\leftarrow Z\rightarrow Y$.
The dimension of the image of the map $\omega_{\check{G}/k}\rightarrow\omega_{JX[p]/k}$
is at most 1 because $h^{0}(\Omega)\leq1$. In particular, the dimension
of the image of $T_{e}JX[p]\rightarrow T_{e}\check{G}$ is at most
1.
\end{proof}
\begin{question}
\label{Question:invariant_forms_nonreduced_pic}Let $X\leftarrow Z\rightarrow Y$
be an étale correspondence of projective curves over $k$ without
a core. Suppose $\text{char}(k)=p$. If $h^{0}(\Omega)=1$, is the
Cartier dual of $Pic^{0}(X\leftarrow Z\rightarrow Y)$ non-reduced?
\end{question}

\section{\label{Section:Clumps_and_bundles}Clumps}
\begin{defn}
Let $X\overset{f}{\leftarrow}Z\overset{g}{\rightarrow}Y$ be a correspondence
of curves over a field $k$. A \emph{clump} $S$ is a finite set of
$\overline{k}$ points $S\subset Z(\overline{k})$ such that $f^{-1}(f(S))=g^{-1}(g(S))=S$.
An \emph{étale} \emph{clump }is a clump $S$ such that $f$ and $g$
are étale at all points of $S$.
\end{defn}
If $X\overset{f}{\leftarrow}Z\overset{g}{\rightarrow}Y$ has a core,
then as in Remark \ref{Remark:Bounded_orbit} every $z\in Z(\overline{k})$
is contained in a clump. In the language of Remark \ref{Remark:Bounded_orbit},
a clump is a finite union of \emph{bounded orbits of geometric points}.

Let $X\overset{f}{\leftarrow}Z\overset{g}{\rightarrow}Y$ be a correspondence
of curves over $k$ of type $(d,e)$. Given an étale clump $S$, we
now construct a natural invariant line bundle $\mathscr{L}(S)$ together
with a one-dimensional subspace $V_{S}\subset H^{0}(\mathscr{L}(S))$
of invariant sections. (This line bundle may only be defined after
a finite extension of $k$.) Think of $S$ as an effective divisor
on $Z$ where all of the coefficients of the points are 1. Then $f_{*}S$
is an effective divisor on $X$, all of whose coefficients are exactly
$d$ because $f$ is étale at all points of $S$ and has degree $d$.
Therefore $\frac{1}{d}f_{*}S$ makes sense as an effective divisor
on $X$; it is the divisor associated to the finite set $f(S)\subset X$.
The associated line bundle $\mathscr{L}_{X}(S)$ on $X$ comes equipped
with a natural one-dimensional space of sections $W_{X}\subset H^{0}(X,\mathscr{L}_{X}(S))$
with the following defining property: for any $w\in W_{X}$, $\text{div}(w)=\frac{1}{d}f_{*}S$.
Moreover, $f^{*}\mathscr{L}_{X}(S)$ is isomorphic to the line bundle
associated with the divisor $S$. Similarly we obtain a line bundle
$\mathscr{L}_{Y}(S)$ on $Y$ with a natural one-dimensional space
of sections $W_{Y}$. We set 
\[
\mathscr{L}(S):=(\mathscr{L}_{X}(S),\mathscr{L}_{Y}(S),\phi)
\]
 for any choice of isomorphism $\phi$ between the pullbacks. The
vector space $H^{0}(\mathscr{L}(S))$ has a natural line $V_{S}$
of invariant sections, given by $f^{*}W_{X}$ and $g^{*}W_{Y}$; in
particular $h^{0}(\mathscr{L}(S))\geq1$.
\begin{cor}
\label{Corollary:char_0_no_clumps}Let $X\leftarrow Z\rightarrow Y$
be a étale correspondence of projective curves over $k$ without a
core. Suppose $\text{char}k=0$. Then there are no clumps.
\end{cor}
\begin{proof}
A clump $S$ is automatically étale and hence yield an nontrivial
invariant line bundle $\mathscr{L}(S)$ such that $h^{0}(\mathscr{L}(S))\geq1$.
This contradicts Corollary \ref{Corollary:no_invariant_sections_char_0}.
\end{proof}
\begin{rem}
Corollary \ref{Corollary:char_0_no_clumps} shows that there is no
direct analog of the supersingular locus in characteristic 0 for the
following reason: Hecke orbits are big. This provides another conceptual
reason why there is no canonical lift for supersingular elliptic curves.
\end{rem}
\begin{cor}
\label{Corollary:cusps_are_ramified}A Hecke correspondence of compactified
modular curves over $\mathbb{C}$ is ramified at at least one of the
cusps.
\end{cor}
\begin{proof}
The cusps are a clump. Hecke correspondences are unramified on open
modular curves; if the compactified correspondence were unramified
at all of the cusps, then the cusps wold form a clump on an étale
correspondence of projective curves without a core, contradicting
Corollary \ref{Corollary:char_0_no_clumps}.
\end{proof}
\begin{rem}
The hypothesis of Corollary \ref{Corollary:char_0_no_clumps} implies
that $X$, $Y$, and $Z$ are Shimura curves by Theorem \ref{Theorem:Mochizuki}.
This corollary was probably known, but we could not find a reference.
Similarly, Corollary \ref{Corollary:cusps_are_ramified} admits a
direct approach, but we find our method conceptually appealing.
\end{rem}
\begin{thm}
\label{Theorem:one_clump}Let $X\overset{f}{\leftarrow}Z\overset{g}{\rightarrow}Y$
be a correspondence of curves over a field $k$ without a core. There
is at most one étale clump.
\end{thm}
\begin{proof}
It is harmless to compactify the correspondence, so we assume $X$,
$Y$, and $Z$ are all projective. Suppose there were two étale clumps,
$S$ and $T$. As in the discussion above, they give rise to positive
invariant line bundles $\mathscr{L}(S)$ and $\mathscr{L}(T)$ together
with lines $V_{S}\subset H^{0}(\mathscr{L}(S))$ and $V_{T}\subset H^{0}(\mathscr{L}(T))$.
There exists $m,n\in\mathbb{N}$ such that $\mathscr{L}(S)^{m}\otimes\mathscr{L}(T)^{-n}$
has degree 0. Lemma \ref{Lemma:No_abelian_subvariety} implies that
$Pic^{0}(X\leftarrow Z\rightarrow Y)$ is a finite group scheme over
$k$; in particular, $\mathscr{L}(S)^{m}\otimes\mathscr{L}(T)^{-n}$
is a torsion line bundle. Therefore there exists $j,k\in\mathbb{N}$
such that $\mathscr{L}(S)^{j}\cong\mathscr{L}(T)^{k}$.

The divisor of any element of $V_{S}^{\otimes j}$ is a positive multiple
of $S$, and similarly the divisor of any element of $V_{T}^{\otimes k}$
is a positive multiple of $T$. In particular, if $S\neq T$, then
the spaces $V_{S}^{\otimes j}$ and $V_{T}^{\otimes k}$ would be
different lines inside of $H^{0}(\mathscr{L}(S)^{j})\cong H^{0}(\mathscr{L}(T)^{k})$.
This would imply that $h^{0}(\mathscr{L}(S)^{j})\geq2$, contradicting
Proposition \ref{Proposition:at_most_one_section}.
\end{proof}
\begin{question}
\label{Question:clump_exist?}Let $k$ be a field of characteristic
$p$. Let $X\overset{f}{\leftarrow}Z\overset{g}{\rightarrow}Y$ be
an étale correspondence of projective curves over $k$ without a core.
Is there always a clump? Equivalently, is there always an invariant
pluricanonical differential form?
\end{question}
\begin{rem}
\label{Remark:generalize_Perret_Hallouin}Theorem \ref{Theorem:one_clump}
generalizes the main theorem of Hallouin and Perret \cite{hallouin2014recursive}
(see the Introduction and Theorem 19 of \emph{loc. cit.}), and the
proof technique is completely different. In particular, they use the
Perron-Frobenius theorem from spectral graph theory. We provide a
detailed description of how to derive their result from ours.

Let $k\cong\mathbb{F}_{q}$ and let $X$ be a smooth projective (geometrically
irreducible) curve over $k$. Hallouin and Perret consider correspondences
$\Gamma\subset X\times X$, with the assumption that $\Gamma$ is
absolutely irreducible and of type $(d,d)$. Let $\mathcal{T}(X,\Gamma)$
be the sequence of curves $(C_{n})_{n\geq1}$ defined as follows:
\[
C_{n}=\{(P_{1},P_{2,}\dots P_{n})\in X^{n}|(P_{i},P_{i+1})\in\Gamma\text{ for each }i=1\dots n-1\}
\]
Let $\mathcal{G}_{\infty}(X,\Gamma)$, the \emph{Geometric Graph},
be the graph whose vertices are the geometric points $X(\mathbb{F})$
and for which there is an oriented edge from $P\in X(\mathbb{F})$
to $Q\in X(\mathbb{F})$ if $(P,Q)\in\Gamma$. Theorem 19 of \emph{loc.
cit.} states that if the $C_{n}$ are irreducible for all $n\geq1$,
then $\mathcal{G}_{\infty}(X,\Gamma)$ has \emph{at most one} finite
$d$-regular subgraph. As the correspondence is of type $(d,d)$,
every finite $d$-regular subgraph of $\mathcal{G}_{\infty}(X,\Gamma)$
induces an étale clump $S_{\Gamma}\subset\Gamma(\mathbb{F})$ with
the following \textquotedbl{}symmetry\textquotedbl{} property: $\pi_{1}(S_{\Gamma})=\pi_{2}(S_{\Gamma})$.
We call $S_{\Gamma}$ a \emph{symmetric étale clump} and set $S_{X}=\pi_{1}(S_{\Gamma})=\pi_{2}(S_{\Gamma})$.

To understand their hypotheses, we first make the following definition.
Let $\Omega$ be an algebraically closed field of transcendence degree
1 over $k$. Let $\mathcal{H}_{gen}^{full}$ be the following \emph{directed
graph}: the vertices are elements of $X(\Omega)$ and the edges are
$\Gamma(\Omega)$. The source of an edge $e$ is $\pi_{1}(e)\in X(\Omega)$
and the target of $e$ is $\pi_{2}(e)\in X(\Omega)$. As usual, this
graph is generally not connected and all connected components are
isomorphic: we let $\mathcal{H}_{gen}$ be any connected component.
Every vertex of the graph $\mathcal{H}_{gen}$ has in-degree and out-degree
$d$. The hypothesis that $C_{n}$ is irreducible for all $n$ is
equivalent to $\mathcal{H}_{gen}$ having \emph{no directed cycles}.
Note that this implies, but is not equivalent to, $\mathcal{H}_{gen}$
being infinite. 

There is of course a surjective ``collapsing'' map $\mathcal{G}_{gen}^{full}\rightarrow\mathcal{H}_{gen}^{full}$
for a self correspondence $X\leftarrow\Gamma\rightarrow X$. One may
make this a map of directed graphs by giving the following orientation
to edges in the 2-colored graph $\mathcal{G}_{gen}^{full}$: an edge
$e$ between a blue vertex $v$ and a red vertex $w$ has the orientation
$v\rightarrow w$. This map\emph{ }does \emph{not necessarily }yield
a surjective map $\mathcal{G}_{gen}\rightarrow\mathcal{H}_{gen}$;
in particular, $\mathcal{G}_{gen}$ can be finite with $\mathcal{H}_{gen}$
infinite\emph{ }(e.g. see Elkies' Example \ref{Example:elkies}.)

We now derive their result from ours. Let us assume, as they implicitly
do, that $\mathcal{H}_{gen}$ has no directed cycles. There are two
options:
\begin{itemize}
\item $X\leftarrow\Gamma\rightarrow X$ has no core (i.e. $\mathcal{G}_{gen}$
is infinite by Proposition \ref{Proposition:no_core_infinite_graph})
\item $X\leftarrow\Gamma\rightarrow X$ has no core (i.e. $\mathcal{G}_{gen}$
is finite by Proposition \ref{Proposition:no_core_infinite_graph})
\end{itemize}
In the first case, Theorem \ref{Theorem:one_clump} directly applies.
In the second case, we will derive their theorem from ours. We first
note that it is sufficient to prove the theorem after replacing $\Gamma$
by its normalization, i.e. we may assume $\Gamma$ is smooth. Call
the coarse core $D$. We have the following diagram.
\[
\xymatrix{ & \Gamma\ar[dl]_{\pi_{1}}^ {}\ar[dr]^{\pi_{2}}\\
X\ar[dr]_{p} &  & X\ar[dl]^{q}\\
 & D
}
\]
As $D$ is the coarse core, $\Gamma$ is the normalization of a component
of $X\times_{p,D,q}X$. A symmetric étale clump $S_{\Gamma}$ of $X\leftarrow\Gamma\rightarrow X$
yields unique étale clump $S_{X}$ for the correspondence $D\leftarrow X\rightarrow D$.
In particular, if we show that $D\leftarrow X\rightarrow D$ has at
most one étale clump, we will have proven $X\leftarrow\Gamma\rightarrow X$
has at most one symmetric étale clump and we will have succeeded in
deriving their theorem from ours.

We need only prove that $D\leftarrow X\rightarrow D$ has no core.
This is where we use the irreducibility of all of the $C_{n}$. Note
that $C_{n}$ is birational to $\Gamma\times_{\pi_{2},X,\pi_{1}}\Gamma\dots\times_{\pi_{2},X,\pi_{1}}\Gamma$
and $\lim_{n\rightarrow\infty}\deg(C_{n}\rightarrow D)=\infty$. On
the other hand, $\Gamma$ is birational to a component of $X\times_{p,D,q}X$.
Therefore $C_{n}$ is birational to an irreducible component
\[
X\times_{p,D,q}X\times\dots\times_{p,D,q}X
\]
with increasing degree over $D$ as $n\rightarrow\infty$. We now
argue this cannot happen if $D\leftarrow X\rightarrow D$ had a core.

If $D\leftarrow X\rightarrow D$ has a core, we can find a curve $W\rightarrow X$
that is finite Galois over both compositions to $D$ by Corollary
\ref{Corollary:Finite_No_Galois}. If $E$ is any irreducible component
of $X\times_{p,D,q}X\times\dots\times_{p,D,q}X$, then 
\[
\deg(E\rightarrow D)\leq\deg(W\rightarrow D)
\]
As the $C_{n}$ are birational to irreducible of components of $X\times_{p,D,q}X\times\dots\times_{p,D,q}X$
and $\deg(C_{n}\rightarrow D)$ goes to $\infty$ as $n\rightarrow\infty$,
we see that $D\leftarrow X\rightarrow D$ has no core. Therefore Theorem
\ref{Theorem:one_clump} applies.

We remark that this argument only requires that there are components
of $C_{n}$ whose degree over $D$ goes to $\infty$ as $n\rightarrow\infty$.
In particular, we only need that $\mathcal{H}_{gen}$ is an infinite
graph. 
\end{rem}
\begin{example}
\label{Example:elkies}Consider the symmetric modular correspondence
$Y(1)\leftarrow Y_{0}(2)\rightarrow Y(1)$ over $\mathbb{F}$. Then
points of the form $\{(P_{1},P_{2},P_{1})|(P_{1},P_{2})\in Y_{0}(2)\}$
are an irreducible component of $C_{3}$. Therefore $C_{3}$ is not
irreducible and their theorem does not directly apply. Note that $\mathcal{G}_{gen}$
is a tree, by direct computation or Lemma \ref{Lemma:free_33}. However,
one can massage the correspondence, à la Elkies \cite{elkies2001explicit},
to obtain the one-clump theorem for this correspondence using their
method: it is equivalent to prove that there is only one clump for
the correspondence $Y_{0}(2)\leftarrow Y_{0}(4)\rightarrow Y_{0}(2)$.
Here $Y_{0}(4)$ parametrizes pairs of elliptic curves equipped with
a cyclic degree 4 isogeny between them $[E_{1}\rightarrow E_{2}]$.
This cyclic isogeny is uniquely the composition $E_{1}\rightarrow E'\rightarrow E_{2}$,
and the two maps to $Y_{0}(2)$ send this isogeny to $[E_{1}\rightarrow E']$
and $[E'\rightarrow E_{2}]$ respectively. Note that this correspondence
has a core: $Y(1)$, where $[E_{1}\rightarrow E']$ and $[E'\rightarrow E_{2}]$
are both sent to $[E']$. Hallouin and Perret's theorem applies to
this correspondence. This correspondence has the property that $\mathcal{G}_{gen}$
is finite (because there is a core) but $\mathcal{H}_{gen}$ is infinite.
For more details, see Hallouin and Perret \cite{hallouin2014recursive}
or Section 2.5 of \cite{krishnamoorthy2016dynamics}.
\end{example}
We describe a simple consequence of having a clump.
\begin{prop}
\label{Proposition:if_clump_then_bounded_field_of_constants}Let $X\leftarrow Z\rightarrow Y$
be a correspondence of curves without a core with $Z$ hyperbolic.
If an étale clump exists, then the degree of the maximal ``field
of constants'' of $E_{\infty}$ is finite over $k$. In other words,
Question \ref{Question:Field_of_Constants} has an affirmative answer. 
\end{prop}
\begin{proof}
If a étale clump exists, then all of the points of the clump are defined
over a finite extension of fields $k'/k$. There are therefore $k'$-valued
points of all of the curves $W_{YX\dots Y}$, as in Remark \ref{Remark:W_infty_and_G_phys}.
This implies that all of the $W_{YX\dots Y}$ and hence $W_{\infty}$
and $E_{\infty}$ have field of constants contained in $k'$. The
field of constants of $E_{\infty}$ is then finite over $k$ as desired.
\end{proof}
\bibliographystyle{plain}
\bibliography{dynamics}

\end{document}